\documentclass[leqno,12pt]{amsart}

\usepackage{amsmath,amssymb,euscript,mathrsfs}

\theoremstyle{plain} \newtheorem*{MMMDefinition}{Definition}
\theoremstyle{plain} \newtheorem*{Theorem1}{Theorem 1}
\theoremstyle{plain} \newtheorem*{ImprovedGS}{Theorem 2}
\theoremstyle{plain} \newtheorem*{Theorem3}{Theorem 3}
\theoremstyle{plain} \newtheorem*{MinorArcsCorollary}{Corollary \ref{MinorArcsCorollary}}
\theoremstyle{plain} \newtheorem*{GSID}{Proposition \ref{GSID}}
\theoremstyle{plain} 
\theoremstyle{plain} \newtheorem*{GS2.1}{Theorem A}
\theoremstyle{plain} \newtheorem*{GS2.2}{Theorem B}
\theoremstyle{plain} 
\theoremstyle{plain} 
\theoremstyle{plain} 
\theoremstyle{plain} 
\theoremstyle{plain} \newtheorem*{theorem2.1}{Theorem 2.1}

\theoremstyle{plain}
\newtheorem{theorem}{Theorem}[section]
\newtheorem{prop}[theorem]{Proposition}
\newtheorem{corollary}[theorem]{Corollary}
\newtheorem{lemma}[theorem]{Lemma}

\newtheorem{conjecture}[theorem]{Conjecture}

\theoremstyle{definition}

\theoremstyle{remark}

\newsymbol\dnd 232D
\newcommand{\dx}{\text{ } d}

\newcommand{\Z}{\mathbb{Z}}
\newcommand{\R}{\mathbb{R}}
\newcommand{\C}{\mathbb{C}}

\newcommand{\U}{\mathbb{U}}
\newcommand{\D}{\mathbb{D}}

\newcommand{\F}{\mathcal{F}}

\newcommand{\s}{\mathcal{S}}
\newcommand{\m}{\mathcal{M}}
\newcommand{\B}{\mathcal{B}}
\newcommand{\K}{\mathcal{K}}

\newcommand{\chimodq}{\chi \, (\text{mod } q)}
\newcommand{\modulo}{\text{mod }}

\numberwithin{equation}{section}  

\begin{document}

\title[Multiplicative Mimicry and the P\'{o}lya-Vinogradov inequality]{Multiplicative mimicry and improvements of the P\'{o}lya-Vinogradov inequality}
\author{Leo Goldmakher}
\date{}

\begin{abstract}
\noindent
We study exponential sums whose coefficients are completely multiplicative and belong to the complex unit disc. Our main result shows that such a sum has substantial cancellation unless the coefficient function is essentially a Dirichlet character. As an application we improve current bounds on odd order character sums. Furthermore, conditionally on the Generalized Riemann Hypothesis we obtain a bound for odd order character sums which is best possible.
\end{abstract}

\maketitle


\section{Introduction}
\label{SectIntro}

Character sums, which encode information on the distribution of primes in arithmetic progressions, have played a central role in the history of analytic number theory. In 1977, on the assumption of the Generalized Riemann Hypothesis, Montgomery and Vaughan \cite{MVExpSumsWMultCoeff} determined an upper bound on character sums which was known to be best-possible for quadratic characters. Recently, under the assumption of the GRH, Granville and Soundararajan \cite{GSPretentiousCharactersAndPV} proved that the Montgomery-Vaughan bound is optimal for characters of every even order. In the same work, they also made breakthroughs in our understanding of odd-order character sums. In the present paper, we develop their ideas further and (again conditionally on the GRH) obtain a best-possible bound on character sums for characters of every odd order, thus completing the story.

Our results on character sums will follow from a more general result, which we discuss first.
Let $\U$ denote the closed complex unit disc $\{|z| \leq 1\}$, and set
\begin{equation}
\label{DefnOfF}
\F = \{f : \Z \to \U \, \Bigl| \, f \text{  is completely multiplicative}\}
\end{equation}
i.e. for all integers $m$ and $n$, $f(mn) = f(m)f(n)$  and $|f(n)| \leq 1$.
Consider the exponential sum
\begin{equation}
\label{PrototypicalExpSumWMultCoeffs}
\sum_{n \leq x} \frac{f(n)}{n} e(n\alpha)
\end{equation}
where $f \in \F$, $\alpha \in \R$, and $e(X) = e^{2 \pi i X}$. By the triangle inequality, this sum has magnitude $\ll \log x$; moreover, this trivial bound is attained in the case $f(n) \equiv 1$ and $\alpha = 0$.\footnote{Here and throughout we use Vinogradov's notation $f \ll g$ to mean $f = O(g)$.} However, the sum cannot in general be this large unless there is a correlation between the behavior of $f(n)$ and $e(n\alpha)$, an unlikely event given that $f$ is completely multiplicative and $e(n\alpha)$ has an additive structure. Perhaps surprisingly, this unlikely scenario does occur non-trivially: taking $f = \chi_{-4}$ (the non-trivial Dirichlet character $(\modulo 4)$) and $\alpha = \frac{1}{4}$, we see that $f(n) = e\bigl(-\frac{1}{4}\bigr) e(n\alpha)$ for all odd integers $n$, from which one can deduce that the magnitude of the exponential sum (\ref{PrototypicalExpSumWMultCoeffs}) is $\gg \log x$. Our first result (Theorem 1) shows that this is essentially the only type of pathological example; precisely, we will show that if the sum has large magnitude, then $f(n)$ must closely mimic the behavior of a function of the form $\xi(n) n^{it}$, where $\xi$ is a Dirichlet character of small conductor and $t$ is a small real number. Moreover, the twist by $n^{it}$ is almost certainly superfluous (see Conjecture \ref{conjecturenotwist}).

Results of this type have been obtained before. In the late 1960s, Hal\'{a}sz \cite{HalaszMeanValMultFns} realized that the mean value of $f \in \F$ is small (in fact, zero) unless $f(n)$ mimics the behavior of a function of the form $n^{it}$. Much more recently, Granville and Soundararajan \cite{GSPretentiousCharactersAndPV} proved that a character sum ${\sum \chi(n)}$ has small magnitude unless $\chi$ mimics the behavior of a Dirichlet character $\xi$ of small conductor and opposite parity. The first part of the present paper is devoted to creating a hybrid of these two methods. When combined with results of Montgomery and Vaughan, this leads to strong bounds on exponential sums of the shape (\ref{PrototypicalExpSumWMultCoeffs}).

Before we can state our main results, we must set up some notation. A common feature in Hal\'{a}sz' and Granville-Soundararajan's work is a measure of how closely one function in $\F$ mimics another. We call this measure the {\em Multiplicative Mimicry Metric}:
\begin{MMMDefinition}[Multiplicative Mimicry Metric]
For any $f,g \in \F$ and any positive $X$, set
\begin{equation}
\label{MMDdefn}
\D(f, g; X) := \left( \sum_{p \leq X} \frac{1 - \text{Re } f(p) \, \overline{g(p)}}{p} \right)^{1/2} .
\end{equation}
\end{MMMDefinition}
\vspace{-.25in}
Note that because $f$ and $g$ are completely multiplicative, their behavior is entirely determined by their values at prime arguments, so the above definition uses all the data on the behavior of $f$ and $g$ (up to $X$). In \cite{GSPretentiousCharactersAndPV}, Granville and Soundararajan observed that this is a pseudometric -- in particular, it satisfies a triangle inequality: $\D(f_1, g_1; X) + \D(f_2, g_2; X) \geq \D(f_1 f_2, g_1 g_2; X)$ for any $f_i, g_i \in \F$. (The only way in which this measure fails to be an honest metric is the possibility that the distance from $f$ to itself might be non-zero.) Further discussions of this pseudometric and some unexpected applications of the triangle inequality can be found in \cite{GSCuriousZetaInequalities}.

Hal\'{a}sz proved that the mean value of a function $f \in \F$ is 0 unless $\D\bigl(f(n), \, n^{it}, \, \infty\bigr) \ll 1$ for some $t \in \R$; moreover, if such a $t$ exists, it is unique. Montgomery \cite{MontgomeryMittagLeffler} and, subsequently, Tenenbaum (III.4.3 of \cite{TenenbaumBook}) found that to further quantify Hal\'{a}sz' result it is convenient to introduce a measure which is closely related to the MM metric:
\begin{equation}
\label{MMmdefn}
\m(f; \, X, \, T) := \min_{|t| \leq T} \; \D\Bigl(f(n), n^{it}; X\Bigr)^2  .
\end{equation}
Essentially, this is measuring how closely $f$ can mimic a function of the form $n^{it}$.
Our main theorem will likewise be stated in terms of this quantity.

For our intended applications, we will need to control the size of the prime factors of the argument. To this end, let $\s(y)$ denote the set of $y$-smooth numbers:
\begin{equation}
\label{DefnOfSmoothNumbers}
\s(y) := \{n \geq 1 : p \leq y \text{ for every prime } p | n\} .
\end{equation}
We can now state a version of our main theorem (for a stronger but more technical statement, see Theorem \ref{HybridBoundThm}):
\begin{Theorem1}
Let $\F$, $\m$, and $\s(y)$ be defined as in (\ref{DefnOfF}), (\ref{MMmdefn}), and (\ref{DefnOfSmoothNumbers}), respectively.
Suppose that $x \geq 2$, $y \geq 16$, $\alpha \in \R$, $f \in \F$,
and that as $\psi$ ranges over all primitive Dirichlet characters of conductor less than $\log y$, $\m(f\overline{\psi}; \, y, \, \log^2 y)$ is minimized when $\psi~=~\xi$.
Then
$$
\mathop{\sum_{n \leq x}}_{n \in \s(y)}
\frac{f(n)}{n} \, e(n\alpha)
\ll
(\log y) \, e^{-\m(f\overline{\xi}; \, y, \, \log^2 y)}
+ (\log y)^{2/3 + o(1)}
$$
where the implicit constant is absolute and $o(1) \to 0$ as $y \to \infty$.
\end{Theorem1}
\vspace{-0.25in}
{\sc Remarks}:
\begin{description}

\vspace{-0.25in}
\item[({\em i})] Colloquially, the theorem asserts that there is lots of cancellation in the exponential sum unless $f(n) \approx \xi(n) \, n^{it}$ for many small $n$, where $\xi$ is some Dirichlet character of small conductor and $t$ is a small real number.
\item[({\em ii})] Formally, the bound is independent of $x$. However, note that for all $y \geq x$ the condition $n \in \s(y)$ becomes superfluous, so if this is the case we can replace all appearances of $y$ by $x$ on the right hand side of the bound.
\item[({\em iii})] As stated, the theorem is uniform in $\alpha$. See Theorem \ref{HybridBoundThm} for a quantitative version which is explicit in the dependence on $\alpha$.
\end{description}

In the second half of this paper we apply the method to the study of character sums. Given a Dirichlet character $\chimodq$, we wish to understand the behavior of the associated character sum function
$$
S_\chi(t) := \sum_{n \leq t} \chi(n) .
$$
The importance of this function is perhaps most easily seen in its intimate connection to the Dirichlet $L$-functions: partial summation on $L(s,\chi)$ leads to the following expression, valid whenever Re $s > 0$:
$$
L(s,\chi) = s \int_1^\infty \frac{1}{t^{s+1}} \, S_\chi(t) \, dt .
$$
In the reverse direction, Perron's formula shows that for any $c > 1$ and any $t \notin \Z$,
$$
S_\chi(t) = \frac{1}{2\pi i} \int_{c-i\infty}^{c+i\infty} L(s,\chi) \, x^s \, \frac{ds}{s} .
$$
The behavior of the character sum function is not well understood, but some progress has been made in studying its magnitude. The first breakthrough occurred in 1918, when P\'{o}lya and Vinogradov independently proved that for all $t$,
\begin{equation}
\label{PVInequality}
| S_\chi(t) | \ll \sqrt{q} \, \log q  .
\end{equation}
This is superior to the trivial bound $|S_\chi(t)| \leq t$ for all $t$ larger than $q^{\frac{1}{2} + \epsilon}$, and is close to being sharp; for all primitive $\chimodq$,
$$
\max_{t \leq q} |S_\chi(t)| \gg \sqrt{q} .
$$
(A slick proof of this is to apply partial summation to the Gauss sum
\begin{equation}
\label{GaussSumDef}
\tau(\chi) := \sum_{n \leq q} \chi(n) \, e\Bigl(\frac{n}{q}\Bigr)
\end{equation}
and use the classical result that for primitive $\chimodq$, $|\tau(\chi)| = \sqrt{q}$.)

The P\'{o}lya-Vinogradov inequality naturally suggests two distinct research goals:
to obtain non-trivial bounds for short character sums, and to improve (\ref{PVInequality}) for long sums. Great progress has been made in the former of these tasks by Burgess, although the current state of knowledge still falls far short of the bound ${|S_\chi(t)| \ll_\epsilon q^\epsilon \sqrt{t}}$ implied by the GRH. The other path, that of sharpening the P\'{o}lya-Vinogradov inequality for long sums, saw little progress until the work of Montgomery and Vaughan \cite{MVExpSumsWMultCoeff}, who proved on the assumption of the GRH that
\begin{equation}
\label{MVInequality}
| S_\chi(t) | \ll \sqrt{q} \, \log \log q  .
\end{equation}
Given the strength of the hypothesis this improvement may seem a bit precious, but in fact it is a best-possible result: in 1932, Paley \cite{Paley} constructed an infinite class of quadratic characters $\{\chi_n \, (\modulo q_n)\}$ for which $$
\max_{t \leq q} |S_{\chi_{{}_n}}(t)| \gg \sqrt{q_n} \, \log \log q_n  .
$$
Unconditionally, however, there were no asymptotic improvements\footnote{There were several improvements of the implicit constant, however. Of particular note is Hildebrand's interesting work \cite{Hildebrand}, wherein he puts forward the idea that $S_\chi(t)$ can only have large magnitude if $\chi$ mimics closely the behavior of a character of very small conductor. It is the development of this idea which led to the work of Granville and Soundararajan, and subsequently to the present paper.} of the P\'{o}lya-Vinogradov inequality for long sums until the recent breakthroughs of Granville and Soundararajan \cite{GSPretentiousCharactersAndPV}. Among other results, they demonstrated that for primitive characters $\chimodq$ of odd order one can unconditionally improve the P\'{o}lya-Vinogradov bound by a power of $\log q$ and, conditionally on the GRH, the Montgomery-Vaughan estimate by a power of $\log \log q$. The following theorem, which will be an immediate consequence of Theorems \ref{SharpThm21} and \ref{GeneralizeLemma32}, improves both Granville-Soundararajan's conditional and unconditional bounds (see the remarks immediately following the theorem).
\begin{ImprovedGS}
For every primitive Dirichlet character $\chimodq$ of odd order $g$,
$$
|S_\chi(t)| \ll_g \sqrt{q} \, (\log Q)^{1 - \delta_g + o(1)}
$$
where $\delta_g := 1 - \frac{g}{\pi} \sin \frac{\pi}{g}$ and
$$
Q =
\begin{cases}
q & \quad \text{unconditionally} \\
\log q & \quad \text{conditionally on the GRH.}
\end{cases}
$$
The implicit constant depends only on $g$, and $o(1) \to 0$ as $q \to \infty$.
\end{ImprovedGS}
\vspace{-0.25in}
{\sc Remarks}:
\begin{description}

\vspace{-0.25in}
\item[({\em i})] Our conditional estimate was conjectured by Granville and Soundararajan in \cite{GSPretentiousCharactersAndPV}.

\item[({\em ii})] $\delta_3 \approx 0.173$, so $1 - \delta_3$ is slightly smaller than $5/6$.

\item[({\em iii})] Theorem 2 saves a factor of $(\log Q)^{\delta_g / 2}$ over the Granville-Soundararajan bounds (see Theorems 1 and 4 of \cite{GSPretentiousCharactersAndPV}).

\item[({\em iv})] The only step in our argument requiring the GRH is Proposition \ref{GRHprop} below.
\end{description}

Finally, we show that the conditional estimate in Theorem 2 is best-possible:
\begin{Theorem3}
Assume the GRH. Then for any odd integer $g \geq 3$, there exists an infinite family of characters $\chimodq$ of order $g$ such that
\[
\max_{t \leq q}| S_\chi(t) | \gg_{\epsilon,g} \sqrt{q} (\log \log q)^{1 - \delta_g - \epsilon}
\]
\end{Theorem3}

In the following section, we state precise versions of our results and outline the arguments which go into proving them.

{\em Acknowledgements:} This work grew out of my Ph.D. thesis, and I am very grateful to my advisor, Soundararajan; it was at his suggestion that I began exploring this interesting subject, and over the past five years he has been extremely generous with his time and support. It is also a pleasure to thank Denis Trotabas and Ilya Baran for helpful discussions, Jeff Lagarias for meticulously reading and commenting on several drafts of this paper, and John Friedlander for his encouragement and for making numerous improvements to the exposition. I would also like to thank the anonymous referees for their thorough reading and insightful comments.

\section{Precise statements of results and sketches of their proofs}
\label{SectDetails}

It has long been understood that cancellation in exponential sums with arithmetic coefficients is closely related to diophantine properties of $\alpha$. To state this more precisely, recall Dirichlet's theorem on diophantine approximation: given any $M \geq 2$ there exists a rational number $\frac{b}{r}$ such that
\begin{equation}
\label{EqDiophantineApproxSect2}
1 \leq r \leq M, \:\: (b,r) = 1,\: \text{ and } \:
\left| \alpha - \frac{b}{r} \right| \leq \frac{1}{rM}  .
\end{equation}
In \cite{MVExpSumsWMultCoeff}, Montgomery and Vaughan showed that there is cancellation in the exponential sum (\ref{PrototypicalExpSumWMultCoeffs}) for $\alpha$ belonging to a `minor arc', i.e. for those $\alpha$ admitting a diophantine approximation by a rational number with large denominator. Our main result complements this by showing that there is substantial cancellation in the sum (\ref{PrototypicalExpSumWMultCoeffs}) even for those $\alpha$ not admitting such a rational approximation, unless both $f(n)$ and $\alpha$ are rather special: $f(n)$ must mimic a function of the form $\xi(n) n^{it}$ for some primitive Dirichlet character $\xi \, (\modulo m)$, and the denominator $r$ of the diophantine approximation for $\alpha$ given by (\ref{EqDiophantineApproxSect2}) must be a multiple of the `exceptional modulus' $m$. Formally:
\begin{theorem}
\label{HybridBoundThm}
Let $\F$, $\m$, and $\s(y)$ be defined as in (\ref{DefnOfF}), (\ref{MMmdefn}), and (\ref{DefnOfSmoothNumbers}), respectively.
Suppose that $x \geq 2$, $y \geq 16$, $\alpha \in \R$, $f \in \F$,
and that as $\psi$ ranges over all primitive Dirichlet characters of conductor less than $\log y$, $\m(f\overline{\psi}; \, y, \, \log^2 y)$ is minimized when $\psi~=~\xi \, (\modulo m)$.
Set $M = \exp\biggl(\exp\Bigl(\frac{\log \log y}{\log \log \log y}\Bigr)\biggr)$.

(I) If there exists $\frac{b}{r}$ satisfying (\ref{EqDiophantineApproxSect2}) with $r > \log y$, then $$
\mathop{\sum_{n \leq x}}_{n \in \s(y)}
\frac{f(n)}{n} \, e(n\alpha)
\ll
(\log y)^{1/2 + o(1)} .
$$
(II) If there exists a rational number of the form $\frac{b}{r}$ such that (\ref{EqDiophantineApproxSect2}) holds with $r \leq \log y$ and $m \nmid r$, then
$$
\mathop{\sum_{n \leq x}}_{n \in \s(y)}
\frac{f(n)}{n} \, e(n\alpha)
\ll
(\log y)^{2/3 + o(1)}  .
$$
(III) If no rational numbers satisfy the hypotheses of (I) or (II), then
$$
\mathop{\sum_{n \leq x}}_{n \in \s(y)}
\frac{f(n)}{n} \, e(n\alpha)
\ll
\frac{\sqrt{m}}{\varphi(m)} \, (\log y) \, e^{-\m(f\overline{\xi}; \, y, \, \log^2 y)}
+ \frac{1}{\sqrt{r}} (\log y)^{2/3 + o(1)} + (\log y)^{1/2 + o(1)}  .
$$
All implicit constants are absolute, and $o(1) \to 0$ as $y \to \infty$.
\end{theorem}
\vspace{-0.25in}
{\sc Remarks}:
\begin{description}

\vspace{-0.25in}
\item[({\em i})] We expect that the twist by $n^{it}$ is superfluous. In other words, taking $\xi \, (\modulo m)$ to be the nearest primitive Dirichlet character to $f(n)$ with respect to the MM metric, the above theorem should hold with $\m(f\overline{\psi}; \, y, \, \log^2 y)$ replaced throughout by $\D(f,\, \psi;\, y)^2$. See Conjecture \ref{conjecturenotwist} and the discussion preceding it for a justification of this belief.

\item[({\em ii})] The methods used to prove Theorem \ref{HybridBoundThm} can be applied to obtain an analogous theorem for sums of the form $\sum f(n) e(n\alpha)$ with $f \in \F$. In this case, in contrast with the previous remark, the twist by $n^{it}$ will be necessary. See the discussion preceding Conjecture \ref{conjecturenotwist}.

\item[({\em iii})] With more work, it should be possible to adapt the argument to prove a similar result under the weaker hypothesis that $f(n)$ is multiplicative (as opposed to completely multiplicative). The hypothesis that $|f(n)| \leq 1$ for all $n$ is much more delicate, however. Proving an analogous result for $f(n)$ whose magnitude grows (however slowly) to infinity would find wide applications, but the methods described here seem insufficient to attack this problem.

\item[({\em iv})] Theorem \ref{HybridBoundThm} immediately implies Theorem 1.
\end{description}

We split the proof into several steps.

\underline{{\sc Step 1: Handling the minor arcs}} \newline
In 1977, Montgomery and Vaughan made an important breakthrough in the study of character sums by proving the upper bound (\ref{MVInequality}) on the assumption of the Generalized Riemann Hypothesis (see \cite{MVExpSumsWMultCoeff}). Most of their paper is devoted to (unconditionally) obtaining cancellation in sums of the form $\sum f(n) e(n\alpha)$ with $f$ multiplicative and $\alpha$ admitting a rational diophantine approximation with large denominator. To accomplish this, they first reduce the problem to studying certain bilinear forms, then develop an intricate iterated version of Dirichlet's hyperbola method to estimate this form. For our purposes, we require a variant of their bound: first, we are interested in sums of the form $\sum \frac{f(n)}{n} e(n\alpha)$, and second, we will need to control the smoothness of the argument. In Section \ref{SectMinorArcs}
we deduce the following from Montgomery and Vaughan's theorem:
\begin{corollary}
\label{MinorArcsCorollary}
Given $f\in\F$, $\alpha \in \R$, and a reduced fraction $\frac{b}{r}$ such that $r \geq 2$ and ${\displaystyle \left|\alpha - \frac{b}{r}\right| \leq \frac{1}{r^2}}$. Then for ${x \geq 2}$ and ${y \geq 16}$,
$$
\mathop{\sum_{n \leq x}}_{n \in \s(y)}
\frac{f(n)}{n} e(n\alpha) \ll
\log r + \frac{(\log r)^{5/2}}{\sqrt{r}} \log y + \log \log y
$$
where the implicit constant is absolute.
\end{corollary}
\vspace{-0.25in}
It is evident that this bound is particularly effective for those $\alpha$ which have a rational Diophantine approximation with large denominator. In the language of the circle method, such $\alpha$ constitute the {\em minor arcs}; all other $\alpha$ (i.e. all of whose rational Diophantine approximations have small denominator) comprise the {\em major arcs}. Thus, Corollary \ref{MinorArcsCorollary} handles the minor arcs, and it remains to tackle those $\alpha$ belonging to a major arc. A method to do this in the case that $f$ is a character was developed by Granville and Soundararajan in \cite{GSPretentiousCharactersAndPV}. In addition to generalizing and streamlining their argument somewhat, we introduce a new ingredient: the work of Hal\'{a}sz, Montgomery, and Tenenbaum on mean values of multiplicative functions. We describe how this is done in the next three steps of our outline.

\underline{{\sc Step 2: The Granville-Soundararajan identity}} \newline
In Section \ref{SectGSIDandReductionToRationals} we prove Lemma \ref{alphatoratl}, which will allow us to replace $\alpha$ by a rational Diophantine approximation in the exponential sum at the cost of possibly shortening the range of summation slightly and adding a negligible error. More precisely, under a weak technical hypothesis (easily satisfied in our situation), it will assert the existence of an $N \leq x$ such that
$$
\mathop{\sum_{n \leq x}}_{n \in \s(y)} \frac{f(n)}{n} \, e(n\alpha) =
\mathop{\sum_{n \leq N}}_{n \in \s(y)} \frac{f(n)}{n} \, e\left(\frac{b}{r} \, n\right) +
O(\log \log y) .
$$
It is worth noting that while our choice of $N$ will be dependent on $\alpha$, the implicit constant in the error term will be absolute.

This step allows us to focus on the case of rational $\alpha$. An identity which is implicit in the work of Granville and Soundararajan (see section 6.2 of \cite{GSPretentiousCharactersAndPV}) gets right to the heart of the matter:
\begin{prop}
[Granville-Soundararajan Identity\footnote{Similar identities appear in work of Hildebrand \cite{Hildebrand} and Montgomery and Vaughan \cite{MVExpSumsWMultCoeff}.}]
\label{GSID}
Given integers $b$ and $r$ such that $(b,r)~=~1$ with $b \neq 0$ and $r \geq 1$. Then for all $f \in \F$, $N \geq 2$, and $y \geq 2$, we have
$$
\mathop{\sum_{n \leq N}}_{n \in \s(y)}
\frac{f(n)}{n} \, e\Bigl(\frac{b}{r} \, n\Bigr)
=
\mathop{\sum_{d|r}}_{d \in \s(y)}
\frac{f(d)}{d} \cdot \frac{1}{\varphi\left(r/d\right)}
\sum_{\psi \, \left(\modulo \frac{r}{d}\right)}
\tau(\psi) \, \overline{\psi}(b)
\left(\mathop{\sum_{n \leq N/d}}_{n \in \s(y)}
\frac{f(n)\overline{\psi}(n)}{n}\right)  .
$$
\end{prop}
\vspace{-0.25in}
The proof can be found in Section \ref{SectGSIDandReductionToRationals}.

In the case that $\alpha$ belongs to a major arc, $r$ will be small, so the only factor on the right hand side which can make a significant contribution is the innermost sum. We thus must turn our attention to sums of the form
$$
\mathop{\sum_{n \leq x}}_{n \in \s(y)} \frac{g(n)}{n}
$$
for $g \in \F$; it is here that we introduce significant refinements into Granville and Soundararajan's ideas.

\underline{{\sc Step 3: A Hal\'{a}sz-like result}} \newline
As mentioned in the introduction, Hal\'{a}sz \cite{HalaszMeanValMultFns} realized that the mean value of $f \in \F$ can be large only if $f(n)$ mimics a function of the form $n^{it}$, where this mimicry is measured by the MM metric. In \cite{MontgomeryMittagLeffler}, Montgomery reworked Hal\'{a}sz' method to bound the magnitude of $\displaystyle \sum_{n \leq x} f(n)$ in terms of the behavior of the generating function of $f$,
\begin{equation}
\label{EqGenSeries}
F(s) := \sum_{n = 1}^\infty \frac{f(n)}{n^s} ,
\end{equation}
in a vertical strip of the complex plane. In $\S$III.4.3 of his excellent book \cite{TenenbaumBook}, Tenenbaum outlines a method of bounding $F(s)$ in terms of the quantity
$$
\m(f; \, X, \, T) := \min_{|t| \leq T} \; \D\Bigl(f(n), n^{it}; X\Bigr)^2  .
$$
In combination with Montgomery's work, this leads to an elegant quantitative version of Hal\'{a}sz' result.

Inspired by Montgomery's reworking of Hal\'{a}sz' method, in 2001 Montgomery and Vaughan \cite{MVMeanValsOfMultFns} bounded
$$
\sum_{n \leq x} \frac{f(n)}{n}
$$
in terms of $F(s)$, the generating series of $f$ defined in (\ref{EqGenSeries}). In Section \ref{SectHalaszTypeBound} we apply Tenenbaum's method to the Montgomery-Vaughan bound to prove the following:
\begin{theorem}
\label{HalMVTenThm}
For $f \in \F$, $x \geq 2$, and $T \geq 1$,
$$
\sum_{n \leq x} \frac{f(n)}{n} \ll (\log x) \, e^{- \m(f; \, x, \, T)} + \frac{1}{\sqrt{T}}
$$
where $\m$ is defined by (\ref{MMmdefn}).
\end{theorem}
\vspace{-0.25in}
From this it is not hard to deduce the following useful corollary.
\begin{corollary}
\label{HalMVTenCor}
For $f \in \F$, $x \geq 2$, $y \geq 2$, and $T \geq 1$,
$$
\mathop{\sum_{n \leq x}}_{n \in \s(y)} \frac{f(n)}{n} \ll (\log y) \, e^{- \m(f; \, y, \, T)} + \frac{1}{\sqrt{T}} .
$$
\end{corollary}
\vspace{-0.25in}
{\sc Remark}: Taking $y = x$ in the corollary immediately yields Theorem \ref{HalMVTenThm}, so the two statements are in fact equivalent.

The above simultaneously refines and generalizes Lemma 4.3 from \cite{GSPretentiousCharactersAndPV}, and is sufficiently strong for our intended application of an optimal bound on odd-order character sums. However, we suspect that more can be said. Colloquially, our bound indicates that $\sum \frac{f(n)}{n}$ can be large only if $f(n)$ mimics a function of the form $n^{it}$. This is an artifact from the proof of the Hal\'{a}sz'-Montgomery-Tenenbaum theorem, which drew the same conclusion for the sum $\sum f(n)$. In that case, $n^{it}$ is an actual enemy since $\sum n^{it}$ is not $o(x)$. Our situation is quite different: if $f(n)$ closely mimics $n^{it}$, then
$$
\sum_{n \leq x} \frac{f(n)}{n} \approx \zeta(1 - it)
$$
which is bounded so long as $t$ is neither too small nor too large. Therefore, for sums of the form considered in Theorem \ref{HalMVTenThm}, $n^{it}$ is no longer an enemy -- the only real enemy is the constant function 1. This leads us to make the following conjecture:
\begin{conjecture}
\label{conjecturenotwist}
For $f \in \F$ and $2 \leq y \leq x$,
$$
\mathop{\sum_{n \leq x}}_{n \in \s(y)} \frac{f(n)}{n} \ll 1+ (\log y) \, e^{- \D(f,\, 1;\, y)^2}  .
$$
\end{conjecture}
\vspace{-0.25in}
Note that the restriction that $y \leq x$ is necessary, as shown by an example of Granville and Soundararajan (directly following Lemma 4.3 of \cite{GSPretentiousCharactersAndPV}).

If some form of this conjecture holds, it would improve our main results (Theorems 1, 2, and \ref{HybridBoundThm}) by removing the possible twist by $n^{it}$, and would allow us to state all the results purely in terms of the distance from $f(n)$ to the nearest primitive character.

\underline{{\sc Step 4: Handling the major arcs}} \newline
One important discovery of Granville and Soundararajan in their study of the MM metric was a repulsion principle similar to the Deuring-Heilbronn phenomenon: $f$ cannot mimic two different characters too closely. Thus, if we identify the `exceptional character' $\xi \, (\modulo m)$ which $f$ most nearly mimics (in the sense made precise in the statement of Theorem \ref{HybridBoundThm}), then $f$ must be quite far from mimicking any other primitive character. In their study of mean values of multiplicative functions in arithmetic progressions \cite{BGS}, Balog, Granville, and Soundararajan derived explicit lower bounds on $\m(f\overline{\psi}; \, y, \, \log^2 y)$ for all primitive $\psi \neq \xi$.

With this in mind, we turn to major arcs. Suppose that $\alpha \approx \frac{b}{r}$ with $r$ small, so that the Montgomery-Vaughan result (Corollary \ref{MinorArcsCorollary}) is not useful. Plugging in the estimate of Corollary \ref{HalMVTenCor} into the right side of the Granville-Soundararajan identity (Proposition \ref{GSID}), we quickly find an upper bound on the magnitude of the left side in terms of the quantities $\m(f\overline{\psi}; \, N/d, \, T)$, where $T$ is a parameter we can specify as we wish and $\psi$ runs over all characters of modulus dividing $r$. If $r$ is not a multiple of the exceptional modulus $m$, then none of the characters $\psi$ are induced by the exceptional character $\xi$; the repulsion principle then implies that $\m(f\overline{\psi}; \, y, \, \log^2 y)$ is bounded from below for all $\psi$ in the sum, meaning that the contribution from each character to the sum is not too large.

If on the other hand $m \mid r$, then some of the characters we are summing over might be induced by the exceptional character $\xi$. In this case, once again using the repulsion principle we can bound $\m(f\overline{\psi}; \, y, \, \log^2 y)$ from below for all $\psi$ which are not induced by $\xi$; however, there will now be a main term coming from the characters induced by the exceptional character. In Section \ref{SectMajorArcs} we make these arguments precise and deduce the following:
\begin{theorem}
\label{MajorArcsThm}
Given $N \geq 2$, $y \geq 16$, $f \in \F$, and $\frac{b}{r}$ a reduced fraction\footnote{We adopt the convention that the reduced form of 0 is $\frac{0}{1}$.} with $1 \leq r \leq \log y$.
Suppose that as $\psi$ ranges over all primitive characters of conductor less than $r$, $\m(f\overline{\psi}; \, y, \, \log^2 y)$ is minimized when $\psi~=~\xi\,(\modulo m)$.
Then
$$
\mathop{\sum_{n \leq N}}_{n \in \s(y)}
\frac{f(n)}{n} \, e\Bigl(\frac{b}{r} n\Bigr)
\ll
\frac{1}{\sqrt{r}} \, (\log y)^{2/3 + o(1)} +
\sqrt{r} \, e^{C \sqrt{\log \log y}} +
\begin{cases}
\frac{\sqrt{m}}{\varphi(m)} \, (\log y) \, e^{-\m(f\overline{\xi}; \, y, \, \log^2 y)}
& \text{  if  } m \mid r \\
0 & \text{otherwise}
\end{cases}
$$
where both $C$ and the implicit constant are absolute, and $o(1) \to 0$ as $y \to \infty$.
\end{theorem}
\vspace{-.25in}
This result is complementary to Corollary \ref{MinorArcsCorollary}, which bounded the same quantity effectively for large $r$; combining the two yields Theorem \ref{HybridBoundThm}, as will be shown in Section \ref{SectCharacterSumBound}.

Having sketched the proof of Theorem \ref{HybridBoundThm}, we move on to sketching the proof of Theorem 2.

\underline{{\sc Application to character sums}} \newline
In their proofs of the P\'{o}lya-Vinogradov inequality (\ref{PVInequality}), both P\'{o}lya and Vinogradov expanded the character sum function $S_\chi(t)$ as a Fourier series (Vinogradov had earlier proved the inequality via other means). P\'{o}lya's version of the Fourier expansion is as follows: for any $N$,
\begin{equation}
\label{PolyasFourierExpansion}
S_\chi(t) = \frac{\tau(\chi)}{2 \pi i}
\sum_{1 \leq |n| \leq N} \frac{\overline{\chi}(n)}{n} \biggl( 1 - e\Bigl( -\frac{nt}{q} \Bigr) \biggr) + O\left(1 + \frac{q\log q}{N}\right)
\end{equation}
where $\tau(\chi)$ denotes the Gauss sum, defined in (\ref{GaussSumDef}).
For any primitive Dirichlet character $\chimodq$, $|\tau(\chi)| = \sqrt{q}$, so we are left to study
sums of the form
\begin{equation}
\label{Eq:CharExpSum}
\sum_{1 \leq |n| \leq N} \frac{\overline{\chi}(n)}{n} \, e(n \alpha) .
\end{equation}
Needless to say, this looks very similar to the sums seen in Theorems 1 and \ref{HybridBoundThm}, aside from $n$ running over both positive and negative values. Actually, we will be able to use this symmetry to our advantage. As a simple illustration of this, we note that if $\chi$ has odd order and $\alpha=0$, the sum (\ref{Eq:CharExpSum}) vanishes.

One important consequence of the GRH is that, for some of the most fundamental sums which occur in multiplicative number theory, the bulk of the contribution comes from the so-called `smooth' arguments, i.e. those with no large prime factors -- see (\ref{DefnOfSmoothNumbers}) above for the precise definition.\footnote{Recall, for example, Littlewood's celebrated result that, on the GRH, $L(1,\chi)$ is well approximated by a short Euler product for any primitive Dirichlet character $\chimodq$. Expanding the product, his result can be roughly written down in the following form: assuming the GRH, $L(1, \chi) \approx \sum_{n \in \s\bigl((\log q)^2\bigr)} \frac{\chi(n)}{n}$.  See \cite{Littlewood} for the original argument, or Section 2 of \cite{GSDistr} for some unconditional versions.} The following proposition is due to Granville and Soundararajan, and is the only step in our argument which depends on the GRH.
\begin{prop}
\label{GRHprop}
Assume the GRH. Then for all primitive Dirichlet characters $\chimodq$ we have
$$
\sum_{n \leq x} \frac{\overline{\chi}(n)}{n} \, e(n\alpha) =
\mathop{\sum_{n \leq x}}_{n \in \s(y)} \frac{\overline{\chi}(n)}{n} \, e(n\alpha) +
O\left(\frac{(\log q)(\log ex)}{y^{1/6}}\right)
$$
uniformly for $1 \leq x \leq q^{3/2}$, $y \geq 1$, and all $\alpha$.
\end{prop}
\vspace{-0.25in}
\begin{proof}
This follows immediately from Lemma 5.2 of \cite{GSPretentiousCharactersAndPV} by partial summation.
\end{proof}
\vspace{-0.25in}
A precursor of this result, with $\alpha=0$, was proved by Montgomery and Vaughan; see Lemma 2 of \cite{MVExpSumsWMultCoeff}.

Very slightly modifying the method used to prove Theorem \ref{HybridBoundThm}, we will show (in Section \ref{SectCharacterSumBound}) that
$$
\sum_{1 \leq |n| \leq q} \frac{\overline{\chi}(n)}{n} \, e(n\alpha)
\ll
\bigl(1 - \chi(-1) \xi(-1)\bigr) \frac{\sqrt{m}}{\varphi(m)} \, (\log Q) \, e^{-\m(\chi\,\overline{\xi}; \, Q, \, \log^2 Q)} + (\log Q)^{2/3 + o(1)}
$$
where the implicit constant is absolute and $o(1) \to 0$ as $q \to \infty$. Colloquially, this indicates that there is a lot of cancellation in the sum on the left hand side unless $\chi(n)$ mimics $\xi(n) \, n^{it}$ for some primitive Dirichlet character $\xi$ of opposite parity and small conductor, and some small real number $t$.

Combining this bound with P\'{o}lya's Fourier expansion (\ref{PolyasFourierExpansion}) we immediately deduce the following:
\begin{theorem}
\label{SharpThm21}
Given a primitive Dirichlet character $\chimodq$, set
$$
Q =
\begin{cases}
q & \quad \text{unconditionally} \\
(\log q)^{12} & \quad \text{conditionally on the GRH.}
\end{cases}
$$
Suppose that as $\psi$ ranges over all primitive characters of conductor less than $\log Q$, $\m(\chi \overline{\psi}; \, Q, \, \log^2 Q)$ is minimized when $\psi~=~\xi \, (\modulo m)$.
Then
$$
\max_{t \leq q} \left|S_\chi(t)\right|
\ll
\bigl( 1 - \chi(-1) \xi(-1) \bigr) \,
\frac{\sqrt{m}}{\varphi(m)} \,
\sqrt{q} \, (\log Q) \, e^{-\m(\chi\overline{\xi}; \, Q, \, \log^2 Q)} +
\sqrt{q} \, (\log Q)^{2/3 + o(1)}
$$
where the implicit constant is absolute and $o(1) \to 0$ as $q \to \infty$.
\end{theorem}
\vspace{-0.25in}
{\sc Remark}:
This refines the main term and sharpens the error term of Theorems 2.1 and 2.4 from Granville and Soundararajan's paper \cite{GSPretentiousCharactersAndPV}.

To conclude the proof of Theorem 2, it remains only to show that given any primitive Dirichlet character $\chimodq$ of odd order, and any primitive character $\xi$ of small conductor and opposite parity, $\chi(n)$ cannot mimic too closely the behavior of $\xi(n) \, n^{it}$ for small $t$. This is reminiscent of Lemma 3.2 of \cite{GSPretentiousCharactersAndPV}, wherein Granville and Soundararajan proved the same statement in the special case that $t=0$. Unfortunately, their argument does not generalize easily, and we are forced to introduce several new ingredients. These are discussed at the beginning of Section \ref{SectLowerBoundOnDistance}, in which we will prove the following:
\begin{theorem}
\label{GeneralizeLemma32}
Given $y\geq 3$, $\chimodq$ a primitive character of odd order $g$, and any odd character $\xi \, (\modulo m)$ with $m < (\log y)^A$. Then
\[
\m(\chi \overline{\xi}; \, y, \, \log^2 y) \geq \bigl(\delta_g + o(1)\bigr) \log \log y
\]
where $o(1) \to 0$ as $y \to \infty$ for any fixed values of $g$ and $A$.
\end{theorem}
\vspace{-0.25in}
Using the bound from Theorem \ref{GeneralizeLemma32} in that of Theorem \ref{SharpThm21}, we deduce Theorem 2.

We conclude the paper with a proof of Theorem 3, which shows that conditionally on the GRH, our bound on odd-order character sums is best possible.

This concludes our outline. We summarize it, more briefly, before carrying out the arguments sketched above. Section \ref{SectMinorArcs} builds on the work of Montgomery and Vaughan estimating the minor arc contributions to the exponential sum $\sum \frac{f(n)}{n} \, e(n \alpha)$, culminating in Corollary \ref{MinorArcsCorollary}. In Section \ref{SectGSIDandReductionToRationals} we prove two elementary results which inform the rest of our arguments: Lemma \ref{alphatoratl} shows that it suffices to consider the case of rational $\alpha$, and an identity of Granville and Soundararajan further reduces the problem to considering a sum of a type previously investigated by Montgomery and Vaughan. In Section \ref{SectHalaszTypeBound} we
apply Tenenbaum's method to Montgomery and Vaughan's bound to obtain Corollary \ref{HalMVTenCor}, a variation on the Hal\'{a}sz-Montgomery-Tenenbaum bound for mean values of multiplicative functions. This puts us in the position to treat the major arcs and prove Theorem \ref{MajorArcsThm}, which we do in Section \ref{SectMajorArcs}. In Section \ref{SectCharacterSumBound} we combine the major arc and minor arc estimates to obtain Theorem \ref{HybridBoundThm}, and subsequently deduce the bound on character sums given by Theorem \ref{SharpThm21}. In Section \ref{SectLowerBoundOnDistance}, we show that a primitive character of odd order cannot mimic too closely any function of the form $\xi(n) n^{it}$, where $\xi$ is a character of even order and small conductor; this is Theorem \ref{GeneralizeLemma32}. Finally, in Section \ref{SectCharLowerBd}, we prove Theorem 3.

\section{The minor arc case: proof of Corollary \ref{MinorArcsCorollary}}
\label{SectMinorArcs}

We begin by recalling a result of Montgomery and Vaughan:
\begin{theorem}[Montgomery-Vaughan]\label{MV1977GSFormulation}
Suppose $f \in \F$ and $|\alpha - b / r| \leq \frac{1}{r^2}$ with $(b,r)=1$. Then for every $R \in [2,r]$ and any $N \geq Rr$ we have
$$
\sum_{Rr \leq n \leq N} \frac{f(n)}{n} \, e(n\alpha) \ll
\log \log N + \frac{(\log R)^{3/2}}{\sqrt{R}} \, \log N
$$
where the implicit constant is absolute.
\end{theorem}
\vspace{-0.25in}
\begin{proof}
This follows immediately from Corollary 1 of \cite{MVExpSumsWMultCoeff} by partial summation; our formulation of this theorem is lifted from Lemma 4.2 of \cite{GSPretentiousCharactersAndPV}.
\end{proof}
Montgomery and Vaughan's proof of the above theorem required both ingenuity and hard analysis, as might be expected in a minor arc estimate. With their result in hand, we can deduce the following corollary (which is modeled on Lemma 6.1 of \cite{GSPretentiousCharactersAndPV}) without much exertion.

\begin{MinorArcsCorollary}
Given $f\in\F$, $\alpha \in \R$, and a reduced fraction $\frac{b}{r}$ such that $r \geq 2$ and ${\displaystyle \left|\alpha - \frac{b}{r}\right| \leq \frac{1}{r^2}}$. Then for ${x \geq 2}$ and ${y \geq 16}$,
$$
\mathop{\sum_{n \leq x}}_{n \in \s(y)}
\frac{f(n)}{n} e(n\alpha) \ll
\log r + \frac{(\log r)^{5/2}}{\sqrt{r}} \log y + \log \log y
$$
where the implicit constant is absolute.
\end{MinorArcsCorollary}
\vspace{-0.25in}
Prior to proving this, we introduce one more piece of notation. Given $f : \Z \to \C$ and any positive number $y$, we define the $y$-smoothed function $f_y$: \label{y-smoothedDefn}
$$
f_y(n) =
\begin{cases}
f(n) \quad & \text{if } n \in \s(y) \\
0 \quad & \text{otherwise.}
\end{cases}
$$
Note that if $f \in \F$, then $f_y \in \F$ as well.

\begin{proof}
The bound is trivially true for $x \leq r^2$, so we assume $x > r^2$.

First, note that for $x \leq y^{\log r}$ the claim follows from Theorem \ref{MV1977GSFormulation} applied to $f_y$:
\begin{eqnarray*}
\mathop{\sum_{n \leq x}}_{n \in \s(y)} \frac{f(n)}{n} e(n\alpha)
& = &
\sum_{n \leq x} \frac{f_y(n)}{n} e(n\alpha) \\
& = &
\sum_{n < r^2} \frac{f_y(n)}{n} e(n\alpha) +
\sum_{r^2 \leq n \leq x} \frac{f_y(n)}{n} e(n\alpha) \\
& \ll &
\log r + \frac{(\log r)^{3/2}}{\sqrt{r}} \, \log x + \log \log x \\
& \ll &
\log r + \frac{(\log r)^{5/2}}{\sqrt{r}} \log y + \log \log y .
\end{eqnarray*}

It therefore suffices to bound
\[
\mathop{\sum_{y^{\log r} < n \leq x}}_{n \in \s(y)} \frac{f(n)}{n} e(n\alpha).
\]
Since $n > y^{\log r}$ if and only if $n > r \cdot n^{1 - \frac{1}{\log y}}$,
\begin{eqnarray*}
\mathop{\sum_{y^{\log r} < n \leq x}}_{n \in \s(y)} \frac{f(n)}{n} e(n\alpha)
& \ll &
\frac{1}{r}
\mathop{\sum_{y^{\log r} < n \leq x}}_{n \in \s(y)} \frac{1}{n^{1 - \frac{1}{\log y}}} \\
& \leq &
\frac{1}{r}
\prod_{p \leq y} \left(1 - \frac{1}{p^{1 - \frac{1}{\log y}}}\right)^{-1} .
\end{eqnarray*}
By the Prime Number Theorem,
$$
\log \; \prod_{p \leq y} \left(1 - \frac{1}{p^{1 - \frac{1}{\log y}}}\right)^{-1}
=
\sum_{p \leq y} \frac{1}{p^{1 - \frac{1}{\log y}}} + O(1)
=
\log \log y + O(1) .
$$
It follows that
$$
\mathop{\sum_{y^{\log r} < n \leq x}}_{n \in \s(y)} \frac{f(n)}{n} e(n\alpha)
\ll \frac{1}{r} \, \log y
$$
and the Corollary is proved.
\end{proof}

\section{Reduction to rational $\alpha$ and the Granville-Soundararajan identity}
\label{SectGSIDandReductionToRationals}

We now begin our approach towards the major arcs. We begin by reducing the problem to the case of rational $\alpha$. The following bound is inspired by Lemma 6.2 of \cite{GSPretentiousCharactersAndPV}:
\begin{lemma}
\label{alphatoratl}
Given $f \in \F$, $\alpha \in \R$, $x \geq 16$, $y \geq 16$ and $M \geq 2$. Suppose the reduced fraction $\frac{b}{r}$ with $r \leq M$ is a rational Diophantine approximation to $\alpha$, i.e.
$$
\left|\alpha - \frac{b}{r}\right| \leq \frac{1}{rM} .
$$
Set $N = \min\left\{x, \frac{1}{|r\alpha - b|}\right\}$. Then for all
$R \in \left[2, \frac{N}{2}\right]$,
$$
\mathop{\sum_{n \leq x}}_{n \in \s(y)} \frac{f(n)}{n} \, e(n\alpha) =
\mathop{\sum_{n \leq N}}_{n \in \s(y)} \frac{f(n)}{n} \, e\left(\frac{b}{r}\, n\right) +
O\left(\log R + \frac{(\log R)^{3/2}}{\sqrt{R}} \, (\log y)^2 + \log \log y \right)
$$
where the implied constant in the error term is absolute.
Moreover, if $M \geq 2(\log y)^4 \log \log y$, the error term above can be replaced by $O(\log \log y)$.
\end{lemma}
\vspace{-0.25in}
\textsc{Remarks:}
\begin{description}

\vspace{-0.25in}
\item[({\em i})] For our intended applications, we will be able to choose an $M$ much larger than $2(\log y)^4 \log \log y$.

\item[({\em ii})] The actual value of $N$ is unimportant; what is important is that $M \leq N \leq x$.
\end{description}
\begin{proof}
If $N = x$ then $\displaystyle \left|\alpha - \frac{b}{r}\right| \leq \frac{1}{rx}$ whence
$$
\mathop{\sum_{n \leq x}}_{n \in \s(y)} \frac{f(n)}{n}
\Biggl(e(n\alpha) - e\left(\frac{b}{r}\, n\right)\Biggr) \ll
\mathop{\sum_{n \leq x}}_{n \in \s(y)} \frac{1}{n} \cdot
n \left|\alpha - \frac{b}{r}\right| \ll 1  .
$$
We therefore assume that $N = \frac{1}{|r\alpha - b|} < x$. Note that this immediately implies that $N \geq M$ and that
$$
\left|\alpha - \frac{b}{r}\right| = \frac{1}{rN}  .
$$
By Dirichlet's theorem, there is a reduced fraction $\frac{b_1}{r_1}$ with $r_1 \leq 2N$ such that
$$
\left|\alpha - \frac{b_1}{r_1}\right| \leq \frac{1}{2 r_1 N}
$$
Note that $\frac{b}{r} \neq \frac{b_1}{r_1}$, since $\left|\alpha - \frac{b_1}{r_1}\right| < \frac{1}{r_1 N}$. Thus,
$$
\frac{1}{r r_1} \leq \left|\frac{b}{r} - \frac{b_1}{r_1}\right| \leq
\frac{1}{2 r_1 N} + \frac{1}{r N}
$$
whence $r_1 \geq N - \frac{r}{2}$. Since $r \leq M \leq N$, we see that
$$
\frac{N}{2} \leq r_1 \leq 2N
$$
so we can trivially bound the (possibly empty) sum
$$
\mathop{\sum_{N < n \leq R r_1}}_{n \in \s(y)} \frac{f(n)}{n} \, e(n\alpha)
\ll
\log \frac{R r_1}{N} = \log R + O(1) .
$$
Once again applying Montgomery-Vaughan's Theorem \ref{MV1977GSFormulation} to $f_y$ (which we can do since $R~\leq~\frac{N}{2}~\leq~r_1$) we see that
\begin{eqnarray*}
\mathop{\sum_{R r_1 < n \leq e^{(\log y)^2}}}_{n \in \s(y)} \frac{f(n)}{n} \, e(n\alpha)
&=&
\sum_{R r_1 < n \leq e^{(\log y)^2}} \frac{f_y(n)}{n} \, e(n\alpha) \\
& \ll &
\log \log y + \frac{(\log R)^{3/2}}{\sqrt{R}} \, (\log y)^2  .
\end{eqnarray*}
Finally, using the same device as in the proof of Corollary \ref{MinorArcsCorollary}, we see that
$$
\mathop{\sum_{e^{(\log y)^2} < n \leq x}}_{n \in \s(y)} \frac{f(n)}{n} \, e(n\alpha)
\ll
\mathop{\sum_{e^{(\log y)^2} < n \leq x}}_{n \in \s(y)} \frac{1}{n}
\ll
\frac{1}{y} \sum_{n \in \s(y)} \frac{1}{n^{1 - \frac{1}{\log y}}}
\ll
1  .
$$
Combining these three bounds, we deduce
$$
\mathop{\sum_{n \leq x}}_{n \in \s(y)} \frac{f(n)}{n} \, e(n\alpha) =
\mathop{\sum_{n \leq N}}_{n \in \s(y)} \frac{f(n)}{n} \, e(n\alpha) +
O\left(1 + \log R + \frac{(\log R)^{3/2}}{\sqrt{R}} \, (\log y)^2 + \log \log y \right) .
$$
Just as at the start of the proof, we have
$$
\mathop{\sum_{n \leq N}}_{n \in \s(y)} \frac{f(n)}{n} \, e(n\alpha) =
\mathop{\sum_{n \leq N}}_{n \in \s(y)} \frac{f(n)}{n} \, e\left(\frac{b}{r}\, n\right) + O(1)
$$
and we conclude the proof of the first part of the theorem.

For the second claim, if $M \geq 2(\log y)^4 \log \log y$, then
$$
r_1 \geq N - \frac{r}{2} \geq M - \frac{M}{2} \geq (\log y)^4 \log \log y .
$$
Taking $R = (\log y)^4 \log \log y$ renders the error $O(\log \log y)$.
\end{proof}

We now suppose we are in the case of rational $\alpha$. The following identity, essentially due to Granville and Soundararajan, highlights the key contributors to the major arcs.
\begin{GSID}[Granville-Soundararajan identity]
Given integers $b$ and $r$ such that $(b,r) = 1$ with $b \neq 0$ and $r \geq 1$. Then for all $f \in \F$, $N \geq 2$, and $y \geq 2$, we have
$$
\mathop{\sum_{n \leq N}}_{n \in \s(y)}
\frac{f(n)}{n} \, e\Bigl(\frac{b}{r} \, n\Bigr)
=
\mathop{\sum_{d|r}}_{d \in \s(y)}
\frac{f(d)}{d} \cdot \frac{1}{\varphi\left(\frac{r}{d}\right)}
\sum_{\psi \, \left(\modulo \frac{r}{d}\right)}
\tau(\psi) \, \overline{\psi}(b)
\left(\mathop{\sum_{n \leq N/d}}_{n \in \s(y)}
\frac{f(n)\overline{\psi}(n)}{n}\right)  .
$$
\end{GSID}
Thus for small $r$, the left hand side can be large only if $\displaystyle \mathop{\sum_{n \leq N/d}}_{n \in \s(y)} \frac{f(n)\overline{\psi}(n)}{n}$ is large for some Dirichlet character $\psi$ of conductor dividing $r$.
\begin{proof}
We examine the left hand side. Summing over all possible greatest common divisors $d$ of $n$ and $r$, and setting $a = n / d$ we find
\begin{equation}
\label{intermed}
\mathop{\sum_{n \leq N}}_{n \in \s(y)}
\frac{f(n)}{n} \, e\!\left(\frac{b}{r} \, n\right)
=
\mathop{\sum_{d|r}}_{d \in \s(y)}
\frac{f(d)}{d}
\mathop{\mathop{\sum_{a \leq \frac{N}{d}}}_{\left(a,\frac{r}{d}\right) = 1}}_{a \in \s(y)}
\frac{f(a)}{a} \, e\!\left(\frac{ab}{r/d} \right)  .
\end{equation}
Now,
$$
e\!\left(\frac{ab}{r/d} \right) =
\sum_{k \, \left(\modulo \frac{r}{d}\right)}
e\!\left(\frac{k}{r/d}\right) \, \delta_{ab}(k)
$$
where $\delta_x$ is the indicator function of $x$. By orthogonality of characters, we can express the indicator function in terms of characters:
$$
\delta_{ab}(k) = \frac{1}{\varphi\left(\frac{r}{d}\right)}
\sum_{\psi \, \left(\modulo \frac{r}{d}\right)} \overline{\psi}(ab) \, \psi(k)
$$
whence, switching the order of summation,
$$
e\!\left(\frac{ab}{r/d} \right) =
\frac{1}{\varphi\left(\frac{r}{d}\right)}
\sum_{\psi \, \left(\modulo \frac{r}{d}\right)} \tau(\psi) \, \overline{\psi}(ab)  .
$$
Plugging this back into (\ref{intermed}) and once again switching order of summation yields
the identity.
\end{proof}

\section{A Hal\'{a}sz-like result: proof of Theorem \ref{HalMVTenThm}}
\label{SectHalaszTypeBound}

Given $f \in \F$, set
$$
F(s) := \sum_{n = 1}^\infty \frac{f(n)}{n^s} .
$$
Note that this generating series converges in the halfplane Re $s > 1$.
\begin{theorem}[Montgomery-Vaughan \cite{MVMeanValsOfMultFns}]
\label{ThmMV2001}
For any $f \in \F$ and $x \geq 3$, we have
$$
\sum_{n \leq x} \frac{f(n)}{n} \ll
\frac{1}{\log x} \int_{\frac{1}{\log x}}^1 \frac{1}{\alpha} \, H(\alpha) \, d\alpha
$$
where
$$
H(\alpha) := \left( \sum_{k \in \Z}
\max_{s \in \B_k(\alpha)} \left| \frac{F(s)}{s-1} \right|^2 \right)^{1/2}
$$
and $\B_k(\alpha)$ is the region in the complex plane defined by
$$
\B_k(\alpha) := \left\{s \in \C : 1 + \alpha \leq \sigma \leq 2
\text{   and   }
|t - k| \leq \frac{1}{2} \right\} .
$$
\end{theorem}
In order to deduce Theorem \ref{HalMVTenThm} from this, we use a bound on $F(s)$ due to Tenenbaum:
\begin{theorem}[Tenenbaum]
\label{CorFBoundFromTenenbaum}
Given $f, F$ as above, $x \geq 3$. Then we have
$$
F(1 + \alpha + it) \ll
\begin{cases}
(\log x) \, e^{-\m(f; \, x, \, T)} & \quad \text{for } |t| \leq T \\
\frac{1}{\alpha} & \quad \text{for } |t| > T
\end{cases}
$$
uniformly for $\alpha \in \left[\frac{1}{\log x}, 1\right]$.
\end{theorem}

We are now in the position to prove Theorem \ref{HalMVTenThm}.
\vspace{-0.25in}
\begin{proof}[Proof of Theorem \ref{HalMVTenThm}]
Applying the bound of Theorem \ref{CorFBoundFromTenenbaum}, we estimate $H(\alpha)$ from Montgomery and Vaughan's Theorem \ref{ThmMV2001} as follows:
\begin{eqnarray*}
H(\alpha) &=& \left( \sum_{k \in \Z}
\max_{s \in \B_k(\alpha)} \left| \frac{F(s)}{s-1} \right|^2 \right)^{1/2} \\
& \leq &
\left( \sum_{k \in \Z} \frac{1}{k^2 + \alpha^2}
\max_{s \in \B_k(\alpha)} | F(s) |^2 \right)^{1/2} \\
& \ll &
(\log x) \, e^{-\m(f; \, x, \, T)} \left( \sum_{|k| \leq T - \frac{1}{2}} \frac{1}{k^2 + \alpha^2}\right)^{1/2} +
\frac{1}{\alpha} \left( \sum_{|k| > T - \frac{1}{2}} \frac{1}{k^2 + \alpha^2}\right)^{1/2} \\
& \ll &
\frac{1}{\alpha} (\log x) \, e^{-\m(f; \, x, \, T)} +
(\log x) \, e^{-\m(f; \, x, \, T)} \left( \sum_{k \leq T} \frac{1}{k^2}\right)^{1/2} +
\frac{1}{\alpha} \left( \sum_{k > T - \frac{1}{2}} \frac{1}{k^2}\right)^{1/2} \\
& \ll &
\frac{1}{\alpha} (\log x) \, e^{-\m(f; \, x, \, T)} + \frac{1}{\alpha \sqrt{T}}
\end{eqnarray*}
Using this bound in Theorem \ref{ThmMV2001} immediately yields the result.
\end{proof}

\begin{proof}[Proof of Corollary \ref{HalMVTenCor}.]
Recall from Section \ref{SectMinorArcs} the convenient notation
$$
f_y(n) :=
\begin{cases}
f(n) \qquad & \text{if } n \in \s(y) \\
0 \qquad & \text{otherwise.}
\end{cases}
$$
As was noted there, $f \in \F$ implies that $f_y \in \F$.
Therefore, by Theorem \ref{HalMVTenThm} we have
$$
\mathop{\sum_{n \leq x}}_{n \in \s(y)} \frac{f(n)}{n} =
\sum_{n \leq x} \frac{f_y(n)}{n} \ll
(\log x) \, e^{-\m(f_y; \, x, \, T)} + \frac{1}{\sqrt{T}} .
$$
The following calculation completes the proof:
\begin{align*}
\m(f_y; \, x, \, T)
& =
\min_{|t| \leq T} \D\bigl(f_y(n), \, n^{it}; \, x\bigr)^2 \\
& =
\min_{|t| \leq T} \sum_{p \leq x} \frac{1 - \text{Re } f_y(p) \, p^{-it}}{p} \\
& =
\min_{|t| \leq T} \left( \sum_{p \leq y} \frac{1 - \text{Re } f(p) \, p^{-it}}{p} +
\sum_{y < p \leq x} \frac{1}{p} \right) \\
& =
\m(f; \, y, \, T) + \log\biggl(\frac{\log x}{\log y}\biggr) + O(1) .
\qedhere
\end{align*}
\end{proof}

\section{The major arc case: proof of Theorem \ref{MajorArcsThm}}
\label{SectMajorArcs}

We first derive the claimed bound for $b \neq 0$. In this case, we can apply the Granville-Soundararajan identity (Proposition \ref{GSID}), which we rewrite in the form
\begin{equation}
\label{EqGSIDrewrite}
\mathop{\sum_{n \leq N}}_{n \in \s(y)}
\frac{f(n)}{n} \, e\Bigl(\frac{b}{r} \, n\Bigr)
=
\mathop{\sum_{d|r}}_{d \in \s(y)}
\frac{f(d)}{d} \, a(d)
\end{equation}
where
$$
a(d) =
\frac{1}{\varphi\left(r/d\right)}
\sum_{\psi \, \left(\modulo \frac{r}{d}\right)}
\tau(\psi) \, \overline{\psi}(b)
\left(\mathop{\sum_{n \leq N/d}}_{n \in \s(y)}
\frac{f(n)\overline{\psi}(n)}{n}\right)  .
$$
Because we are assuming $r < \log y$, the restriction $d \in \s(y)$ above is superfluous.

Our first goal is to identify the {\em exceptional character}, the one primitive character which is the primary contributor to our exponential sum. To this end, consider the set of all primitive characters with conductor not exceeding $r$, where we include the constant function {\bf 1} as the primitive character $(\modulo 1)$ which induces all the principal characters to larger moduli. Enumerate all of these primitive characters as $\psi_k \, (\modulo m_k)$ in such a way that
$$
\m(f\overline{\psi_1};\, y,\, \log^2 y) \leq
\m(f\overline{\psi_2};\, y,\, \log^2 y) \leq
\ldots
$$
It will be seen that $\psi_1 \, (\modulo m_1)$ is the exceptional character for $f$; this is the character we called $\xi \, (\modulo m)$ in the statement of the theorem, and its contribution to the sum is difficult to control. We will return to this point later in the proof.

The behavior of the characters $\left(\modulo \frac{r}{d}\right)$ is determined by the set of primitive characters inducing them, so for ease of reference we define for each $d | r$ the set
$$
\K_d = \biggl\{k \, : \, m_k \, \biggl| \, \frac{r}{d} \biggr\} .
$$
Note that $\left|\K_d\right| = \varphi\bigl(\frac{r}{d}\bigr)$. We can rewrite $a(d)$ in terms of the underlying primitive characters $\bigl\{\psi_k \, (\modulo m_k)\bigr\}_{k \in \K_d}$:
$$
a(d) = \frac{1}{\varphi\Bigl(\frac{r}{d}\Bigr)} \sum_{k \in \K_d} \tau(\psi_k \chi_{{}_0}) \,
\overline{\psi_k}(b) \, \chi_{{}_0}(b)
\left(\mathop{\sum_{n \leq N/d}}_{n \in \s(y)}
\frac{f(n) \, \overline{\psi_k}(n) \, \chi_{{}_0}}{n}\right)
$$
where $\chi_{{}_0}$ is the principal character $\bigl(\modulo \frac{r}{d}\bigr)$.
A straightforward calculation shows that if a character $\psi \, (\modulo m)$ is induced by the primitive character $\psi^* \, (\modulo m^*)$, then
$$
\tau(\psi) = \mu\Bigl(\frac{m}{m^*}\Bigr) \, \psi^*\Bigl(\frac{m}{m^*}\Bigr) \, \tau(\psi^*) .
$$
Therefore,
$$
a(d) =
\frac{\chi_{{}_0}(b)}{\varphi\bigl(\frac{r}{d}\bigr)}
\sum_{k \in \K_d} \mu\Bigl(\frac{r}{d m_k}\Bigr) \,
\psi_k\Bigl(\frac{r}{d m_k}\Bigr) \, \tau(\psi_k) \, \overline{\psi_k}(b)
\mathop{\mathop{\sum_{n \leq N/d}}_{n \in \s(y)}}_{\bigl(n, \frac{r}{d}\bigr) = 1}
\frac{f(n) \, \overline{\psi_k}(n)}{n}
$$
We make one final cosmetic adjustment prior to estimating this quantity.
Hildebrand proved the following useful result (see Lemma 5 of \cite{Hildebrand}): for any $g \in \F$ and $x \geq 1$,
$$
\mathop{\sum_{n \leq x}}_{(n,k) = 1} \frac{g(n)}{n} =
\prod_{p | k} \left(1 - \frac{g(p)}{p}\right) \sum_{n \leq x} \frac{g(n)}{n} +
O\Bigl(\bigl(\log \log (k + 2)\bigr)^3\Bigr)
$$
where the implicit constant is absolute.\footnote{See Lemma 4.4 of \cite{GSPretentiousCharactersAndPV} for a substantially similar result.} Set $g = f \overline{\psi}$ for any Dirichlet character $\psi$, and let $g_y$ be the $y$-smoothed version of $g$ (defined on page \pageref{y-smoothedDefn}). Applying Hildebrand's lemma to $g_y$ and using the inequalities $d \leq r \leq y$, we see that
\begin{eqnarray*}
\mathop{\mathop{\sum_{n \leq N/d}}_{n \in \s(y)}}_{\bigl(n, \frac{r}{d}\bigr) = 1}
\frac{f(n) \, \overline{\psi}(n)}{n}
& = &
\mathop{\sum_{n \leq N/d}}_{\bigl(n, \frac{r}{d}\bigr) = 1} \frac{g_y(n)}{n} \\
& = &
\mathop{\sum_{n \leq N}}_{\bigl(n, \frac{r}{d}\bigr) = 1} \frac{g_y(n)}{n} + O(\log d) \\
& = &
\prod_{p | \frac{r}{d}} \left(1 - \frac{g_y(p)}{p}\right) \sum_{n \leq N} \frac{g_y(n)}{n} +
O(\log r) \\
& = &
\prod_{p | \frac{r}{d}} \left(1 - \frac{f(p) \, \overline{\psi}(p)}{p}\right)
\mathop{\sum_{n \leq N}}_{n \in \s(y)} \frac{f(n) \, \overline{\psi}(n)}{n} + O(\log r)
\end{eqnarray*}
Therefore, continuing our calculation from above,
$$
a(d) =
\frac{\chi_{{}_0}(b)}{\varphi\bigl(\frac{r}{d}\bigr)}
\sum_{k \in \K_d} \mu\Bigl(\frac{r}{d m_k}\Bigr) \,
\psi_k\Bigl(\frac{r}{d m_k}\Bigr) \, \tau(\psi_k) \, \overline{\psi_k}(b)
\prod_{p | \frac{r}{d}} \left(1 - \frac{f(p) \, \overline{\psi_k}(p)}{p}\right)
\mathop{\sum_{n \leq N}}_{n \in \s(y)} \frac{f(n) \, \overline{\psi_k}(n)}{n}
$$
up to an error of size
\begin{equation}
\label{EqErrorTermForad}
\ll \frac{1}{\varphi\bigl(\frac{r}{d}\bigr)} \sum_{k \in \K_d} \sqrt{m_k} \, \log r
\ll \sqrt{\frac{r}{d}} \, \log r
\end{equation}
since $m_k \bigl| \frac{r}{d}$ and $|\K_d| = \varphi\bigl(\frac{r}{d}\bigr)$. Before further refining our estimate for $a(d)$, we bound the accumulation of the error (\ref{EqErrorTermForad}) in the sum
$$
\mathop{\sum_{n \leq N}}_{n \in \s(y)}
\frac{f(n)}{n} \, e\Bigl(\frac{b}{r} \, n\Bigr)
=
\sum_{d|r} \frac{f(d)}{d} \, a(d)  .
$$
Since $r < \log y$, we find that the total possible contribution from the error terms is
\begin{equation}
\label{EqTotContribFromError}
\ll
\sum_{d | r} \frac{1}{d} \sqrt{\frac{r}{d}} \, \log r
\ll
\sqrt{r} \, \log r
\ll
\sqrt{r} \, \log \log y .
\end{equation}
In view of the bound claimed in Theorem \ref{MajorArcsThm}, this is negligible.

We now show that the contribution from all the non-exceptional characters $\displaystyle \{\psi_k\}_{k \geq 2}$ to $a(d)$ is not terribly large. From Corollary \ref{HalMVTenCor} we deduce that
\begin{multline}
\nonumber
\frac{\chi_{{}_0}(b)}{\varphi\bigl(\frac{r}{d}\bigr)}
\mathop{\sum_{k \in \K_d}}_{k \geq 2}
\mu\Bigl(\frac{r}{d m_k}\Bigr) \, \psi_k\Bigl(\frac{r}{d m_k}\Bigr) \, \tau(\psi_k) \, \overline{\psi_k}(b) \prod_{p | \frac{r}{d}} \left(1 - \frac{f(p) \, \overline{\psi_k}(p)}{p}\right)
\mathop{\sum_{n \leq N}}_{n \in \s(y)} \frac{f(n) \, \overline{\psi_k}(n)}{n} \ll \\
\ll
\frac{1}{\varphi\bigl(\frac{r}{d}\bigr)}
\mathop{\sum_{k \in \K_d}}_{k \geq 2}
\sqrt{m_k} \, \left(\prod_{p | \frac{r}{d}} \left(1 + \frac{1}{p}\right)\right)
\Bigl((\log y) \, e^{-\m(f \, \overline{\psi_k}; \, y, \, \log^2 y)} + \frac{1}{\log y}\Bigr) .
\end{multline}
Note that for any $g \in \F$ and any $T \geq 0$ we have $0 \leq \m(g; \, y, \, T) \leq 2 \log \log y + O(1)$, whence
$$
(\log y) \, e^{-\m(f \, \overline{\psi_k}; \, y, \, \log^2 y)} \gg \frac{1}{\log y} .
$$
Also, $m_k \leq \frac{r}{d}$ for all $k \in \K_d$, and
$$
\prod_{p | \frac{r}{d}} \left(1 + \frac{1}{p}\right) \ll \log \log \left(\frac{r}{d}+2\right)  .
$$
Therefore, the contribution from all the $k \geq 2$ to $a(d)$ is
$$
\ll
\frac{1}{\varphi\bigl(\frac{r}{d}\bigr)} \,
\sqrt{\frac{r}{d}} \, \biggl(\log \log \Bigl(\frac{r}{d}+2\Bigr)\biggr) \, (\log y)
\mathop{\sum_{k \in \K_d}}_{k \geq 2}
e^{-\m(f \, \overline{\psi_k}; \, y, \, \log^2 y)}  .
$$
To make further progress, we need lower bounds on $\m(f \, \overline{\psi_k}; \, y, \, \log^2 y)$ for $k \geq 2$; in other words, we wish to show that $f(n)$ cannot mimic too closely a function of the form $\psi(n) n^{it}$ so long as $\psi$ is not induced by the exceptional character $\psi_1$.
Fortuitously, such bounds were determined by Balog, Granville, and Soundararajan in their recent study of mean values of multiplicative functions over arithmetic progressions \cite{BGS}. Lemma 3.3 of that paper asserts that for all $k \geq 2$,
\begin{equation}
\label{EqBGSLowerBdOneThird}
\m(f \, \overline{\psi_k}; \, y, \, \log^2 y) \geq \left(\frac{1}{3} + o(1)\right) \log \log y .
\end{equation}
For larger values of $k$ we can do even better: from Lemma 3.1 of \cite{BGS} we deduce that
for all $k > \sqrt{\log \log y}$,
$$
\m(f \, \overline{\psi_k}; \, y, \, \log^2 y) \geq \log \log y + O\bigl(\sqrt{\log \log y}\bigr)  .
$$
Using these bounds in our calculations above (and keeping in mind that $|\K_d| = \varphi\bigl(\frac{r}{d}\bigr)$) we find that the contribution to $a(d)$ from all those $k \geq 2$ which are in $\K_d$ is
$$
\ll
\frac{1}{\varphi\bigl(\frac{r}{d}\bigr)} \,
\sqrt{\frac{r}{d}} \, \biggl(\log \log \Bigl(\frac{r}{d}+2\Bigr)\biggr) \, (\log y)^{2/3 + o(1)} +
\sqrt{\frac{r}{d}} \, \biggl(\log \log \Bigl(\frac{r}{d}+2\Bigr)\biggr) \, e^{O(\sqrt{\log \log y})} .
$$
Going back to equation (\ref{EqGSIDrewrite}), we see that the total contribution of all such terms to $$
\mathop{\sum_{n \leq N}}_{n \in \s(y)}
\frac{f(n)}{n} \, e\Bigl(\frac{b}{r} \, n\Bigr)
=
\sum_{d|r} \frac{f(d)}{d} \, a(d)
$$
is
\begin{eqnarray*}
&\ll &
\sum_{d|r} \frac{1}{d} \left(
\frac{1}{\varphi\bigl(\frac{r}{d}\bigr)} \,
\sqrt{\frac{r}{d}} \, \biggl(\log \log \Bigl(\frac{r}{d}+2\Bigr)\biggr) \, (\log y)^{2/3 + o(1)} +
\sqrt{\frac{r}{d}} \, \biggl(\log \log \Bigl(\frac{r}{d}+2\Bigr)\biggr) \, e^{O(\sqrt{\log \log y})}
\right) \\
& \ll &
\sqrt{r} \, \bigl(\log \log (r+2)\bigr) \sum_{d|r} \left(\frac{1}{d}\right)^{3/2} \left(\frac{1}{\varphi\bigl(\frac{r}{d}\bigr)} \, (\log y)^{2/3 + o(1)}
+ e^{O(\sqrt{\log \log y})}\right) \\
& \ll &
\frac{1}{r} \, \bigl(\log \log (r+2)\bigr) \, (\log y)^{2/3 + o(1)}
\sum_{d|r} \frac{d^{3/2}}{\varphi(d)} +
\sqrt{r} \, \bigl(\log \log (r+2)\bigr) e^{O(\sqrt{\log \log y})}
\end{eqnarray*}
where we have used the change of variables $d \leftrightarrow \frac{r}{d}$ in the sum.
Finally, recall that
$$
\frac{n}{\varphi(n)} \ll \log \log n \qquad \text{and} \qquad \log d(n) \ll \frac{\log n}{\log \log n}
$$
where $d(n)$ denotes the number of divisors of $n$; in particular, we deduce that ${d(r) \ll (\log y)^{o(1)}}$ where $o(1) \to 0$ as $y \to \infty$. Using these bounds in conjunction with our above results, we deduce that the total contribution of all the primitive characters $\psi_k$ with $k \geq 2$ is
$$
\ll
\frac{1}{\sqrt{r}} \, (\log y)^{2/3 + o(1)}
+
\sqrt{r} \, e^{C\sqrt{\log \log y}}
$$
where both $C$ and the implicit constant are absolute, and $o(1) \to 0$ as $y \to \infty$.

If $m_1 \nmid r$ then $1 \not\in \K_d$ for all $d \mid r$, which means that the exceptional character $\psi_1 \, (\modulo m_1)$ does not contribute anything to our exponential sum. In this case, our above estimates tell the whole story, and we conclude the proof of the theorem.

Now suppose instead that $m_1 \mid r$; in this case, we must estimate the contribution from the exceptional character $\psi_1 \, (\modulo m_1)$ to each $a(d)$. This character appears in our sum precisely whenever $1 \in \K_d$ (i.e. whenever $\psi_1$ induces a character $\bigl(\modulo \frac{r}{d}\bigr)$), so the total contribution of this exceptional character is
$$
\sum_{d \, \bigl| \frac{r}{m_1}}
\frac{f(d)}{d} \cdot
\frac{1}{\varphi\bigl(\frac{r}{d}\bigr)} \,
\mu\Bigl(\frac{r}{dm_1}\Bigr) \,
\psi_1\Bigl(\frac{r}{dm_1}\Bigr) \,
\tau(\psi_1) \,
\overline{\psi_1}(b)
\left( \prod_{p \, \bigl| \frac{r}{dm_1}} \left(1 - \frac{f \, \overline{\psi_1}(p)}{p}\right) \right)
\mathop{\sum_{n \leq N}}_{n \in \s(y)} \frac{f \, \overline{\psi_1}(n)}{n}  .
$$
Note that the product now runs over only those $p$ dividing $\frac{r}{dm_1}$, not just those dividing $\frac{r}{d}$ (it is easily seen that this extra restriction does not change the value of the product).
Making the change of variables $d \leftrightarrow \frac{r}{dm_1}$, we find that $\psi_1$'s contribution can be rewritten in the form
\begin{equation}
\label{EqMainTermContrib}
\frac{m_1}{r} \,
\tau(\psi_1) \,
\overline{\psi_1}(b) \,
\left( \mathop{\sum_{n \leq N}}_{n \in \s(y)} \frac{f \, \overline{\psi_1}(n)}{n} \right)
\sum_{d \, \bigl| \frac{r}{m_1}} f\Bigl(\frac{r}{dm_1}\Bigr) A(d)
\end{equation}
where
$$
A(d) =
\frac{d}{\varphi(dm_1)} \,
\mu(d) \,
\psi_1(d) \,
\prod_{p|d} \left(1 - \frac{f \, \overline{\psi_1}(p)}{p}\right)  .
$$
Note that $A(d) = 0$ whenever $(d,m_1) \neq 1$, so only those $d$ which are coprime to $m_1$ contribute to the sum in (\ref{EqMainTermContrib}). Moreover, the same reasoning shows that we need only consider squarefree $d$. Therefore,
\begin{eqnarray*}
A(d)
& = &
\frac{1}{\varphi(m_1)}
\cdot
\frac{d \, \mu(d) \, \psi_1(d)}{\varphi(d)}
\prod_{p|d} \left(1 - \frac{f \, \overline{\psi_1}(p)}{p}\right) \\
& = &
\frac{1}{\varphi(m_1)}
\prod_{p|d} \left(\frac{f(p) - \psi_1(p) \cdot p}{\varphi(p)}\right) \\
& \ll &
\frac{1}{\varphi(m_1)}
\prod_{p|d} \left(\frac{p+1}{p-1}\right) \\
& \ll &
\frac{1}{\varphi(m_1)}
\big( \log \log (d+2) \bigr)^2 .
\end{eqnarray*}
Combining this with Corollary \ref{HalMVTenCor} and (\ref{EqMainTermContrib}) and making elementary estimates as above, we conclude that the total contribution from $\psi_1$ is
$$
\ll \frac{\sqrt{m_1}}{\varphi(m_1)} \, (\log y) \,
e^{-\m(f \, \overline{\psi_1}; \, y, \, \log^2y)} ;
$$
this completes the proof of Theorem \ref{MajorArcsThm} in the case $b \neq 0$.

To show that the same bound holds for the case $b=0$, we consider two separate cases: either $\psi_1$ is the trivial character {\bf 1}, or it isn't. In the former scenario, $m_1 = 1$, so from Corollary \ref{HalMVTenCor} we deduce that
\begin{equation}
\label{EqXiIsTrivial}
\mathop{\sum_{n \leq N}}_{n \in \s(y)} \frac{f(n)}{n}
\ll
\frac{\sqrt{m_1}}{\varphi(m_1)} \, e^{-\m(f \, \overline{\psi_1}; \, y, \, \log^2y)} .
\end{equation}
If, on the other hand, $\psi_1$ is not the trivial character, then by Corollary \ref{HalMVTenCor} together with the lower bound (\ref{EqBGSLowerBdOneThird}) we find
\begin{equation}
\label{EqXiIsNotTrivial}
\mathop{\sum_{n \leq N}}_{n \in \s(y)} \frac{f(n)}{n}
\ll
\frac{1}{\sqrt{r}} \, (\log y)^{2/3 + o(1)}
\end{equation}
(recall our convention that the reduced form of 0 is $\frac{0}{1}$, so $r=1$). In either case, these bounds are subsumed by those claimed. This concludes the proof.

\section{Exponential sums with multiplicative coefficients and character sums: proofs of Theorems \ref{HybridBoundThm} and \ref{SharpThm21}}
\label{SectCharacterSumBound}

Having dealt with both the major and minor arcs, we can now prove Theorem~\ref{HybridBoundThm} without too much difficulty.

\begin{proof}[Proof of Theorem \ref{HybridBoundThm}]
As in the statement of the theorem, set $M = \exp\biggl(\exp\Bigl(\frac{\log \log y}{\log \log \log y}\Bigr)\biggr)$. By Dirichlet's theorem on Diophantine approximation, there exists a reduced fraction $\frac{b}{r}$ with $1 \leq r \leq M$, such that
\begin{equation}
\label{EqDiophantineApprox}
\left| \alpha - \frac{b}{r} \right| \leq \frac{1}{rM} .
\end{equation}
If the hypotheses of (I) hold (i.e. if $\alpha$ belongs to a minor arc), Corollary \ref{MinorArcsCorollary} immediately implies the result claimed.

Suppose instead that the hypotheses of (I) fail to hold (i.e. $\alpha$ belongs to a major arc). By Lemma \ref{alphatoratl}, since $M \geq 2(\log y)^4 \, \log \log y$ there exists an
$N~\in~[M, x]$ such that
$$
\mathop{\sum_{n \leq x}}_{n \in \s(y)} \frac{f(n)}{n} \, e(n\alpha)
=
\mathop{\sum_{n \leq N}}_{n \in \s(y)} \frac{f(n)}{n} \, e\Bigl(\frac{b}{r} \, n\Bigr) +
O(\log \log y) .
$$
Applying Theorem \ref{MajorArcsThm} immediately yields the claim for the scenarios (II) and (III).
\end{proof}

Theorem \ref{SharpThm21} is not much harder:
\begin{proof}[Proof of Theorem \ref{SharpThm21}]
Taking $N=q$ in P\'{o}lya's Fourier expansion (\ref{PolyasFourierExpansion}) we see that
we must bound the sum
$$
\sum_{1 \leq |n| \leq q} \frac{\overline{\chi}(n)}{n} \, e(n\alpha)
$$
for $\alpha = 0$ or $-\frac{nt}{q}$. As in the proof of Theorem \ref{MajorArcsThm}, we treat the cases $\alpha = 0$ and $\alpha \neq 0$ separately, starting with the latter.

Recall from the statement of the theorem that we set
$$
Q =
\begin{cases}
q & \quad \text{unconditionally} \\
(\log q)^{12} & \quad \text{conditionally on the GRH.}
\end{cases}
$$
We use Proposition \ref{GRHprop} to restrict attention to smooth arguments, in the case that the GRH is assumed:
\begin{equation}
\label{EqSmoothRestriction}
\sum_{1 \leq |n| \leq q} \frac{\overline{\chi}(n)}{n} \, e(n\alpha)
=
\mathop{\sum_{1 \leq |n| \leq q}}_{n \in \s(Q)} \frac{\overline{\chi}(n)}{n} \, e(n\alpha)
+ O(1) .
\end{equation}
Note that this holds unconditionally as well, albeit with superfluous error term. We next find a Diophantine rational approximation to $\alpha$, i.e. a reduced fraction $\frac{b}{r}$ with $1 \leq r \leq M$ such that
$$
\left| \alpha - \frac{b}{r} \right| \leq \frac{1}{rM}  .
$$
Lemma \ref{alphatoratl} asserts that for $M \geq 2(\log Q)^4 \log \log Q$ there exists $N \in [M, q]$ such that
$$
\mathop{\sum_{1 \leq |n| \leq q}}_{n \in \s(Q)} \frac{\overline{\chi}(n)}{n} \, e(n\alpha)
=
\mathop{\sum_{1 \leq |n| \leq N}}_{n \in \s(Q)} \frac{\overline{\chi}(n)}{n} \, e\Bigl(\frac{b}{r} \, n\Bigr) + O(\log \log Q)  .
$$
Finally, note that
$$
\mathop{\sum_{1 \leq |n| \leq N}}_{n \in \s(Q)} \frac{\overline{\chi}(n)}{n} \, e\Bigl(\frac{b}{r} \, n\Bigr)
=
\mathop{\sum_{n \leq N}}_{n \in \s(Q)} \frac{\overline{\chi}(n)}{n} \, e\Bigl(\frac{b}{r} \, n\Bigr)
- \chi(-1)
\mathop{\sum_{n \leq N}}_{n \in \s(Q)} \frac{\overline{\chi}(n)}{n} \, e\Bigl(\frac{-b}{r} \, n\Bigr)  .
$$
Since $M \to \infty$ with $q$ while $\alpha \neq 0$ remains fixed, we must have $b \neq 0$. It follows that we can apply the Granville-Soundararajan identity (Property \ref{GSID}) to both of the expressions on the right hand side of the above equation, deducing the relation
\begin{multline} \nonumber
\mathop{\sum_{1 \leq |n| \leq N}}_{n \in \s(Q)} \frac{\overline{\chi}(n)}{n} \, e\Bigl(\frac{b}{r} \, n\Bigr) = \\
= \mathop{\sum_{d|r}}_{d \in \s(Q)}
\frac{\overline{\chi}(d)}{d} \cdot \frac{1}{\varphi\left(\frac{r}{d}\right)}
\sum_{\psi \, \left(\modulo \frac{r}{d}\right)}
\bigl(1 - \chi(-1) \psi(-1)\bigr) \,
\tau(\overline{\psi}) \, \psi(b)
\left(\mathop{\sum_{n \leq N/d}}_{n \in \s(Q)}
\frac{\overline{\chi}(n) \psi(n)}{n}\right)  .
\end{multline}
The arguments from the proofs of Theorems \ref{HybridBoundThm} and \ref{MajorArcsThm} carry over virtually verbatim, and we conclude that for $\alpha \neq 0$,
$$
\sum_{1 \leq |n| \leq q} \frac{\overline{\chi}(n)}{n} \, e(n\alpha)
\ll
\bigl(1 - \chi(-1) \xi(-1)\bigr) \frac{\sqrt{m}}{\varphi(m)} \, (\log Q) \, e^{-\m(\chi\,\overline{\xi}; \, Q, \, \log^2 Q)} + (\log Q)^{2/3 + o(1)}
$$
where the implicit constant is absolute, and $o(1) \to 0$ as $q \to \infty$.

We now treat the case $\alpha = 0$; again, the arguments will be familiar. We begin as before, by using (\ref{EqSmoothRestriction}) to (potentially) restrict the sum
$$
\sum_{1 \leq |n| \leq q} \frac{\overline{\chi}(n)}{n}
=
\bigl(1 - \chi(-1)\bigr) \sum_{n \leq q} \frac{\overline{\chi}(n)}{n}
$$
to $Q$-smooth arguments. We consider separately the two cases $\xi = {\textbf 1}$ and $\xi \neq {\textbf 1}$. In the former, $\xi(-1) = 1$, whence
\begin{eqnarray*}
\bigl(1 - \chi(-1)\bigr) \mathop{\sum_{n \leq q}}_{n \in \s(Q)} \frac{\overline{\chi}(n)}{n}
&=&
\bigl(1 - \chi(-1)\xi(-1)\bigr) \mathop{\sum_{n \leq q}}_{n \in \s(Q)} \frac{\xi(n)\overline{\chi}(n)}{n} \\
&\ll&
\bigl(1 - \chi(-1)\xi(-1)\bigr) \frac{\sqrt{m}}{\varphi(m)} \, e^{-\m(\chi\,\overline{\xi}; \, Q, \, \log^2 Q)}
\end{eqnarray*}
by Corollary \ref{HalMVTenCor} (as in (\ref{EqXiIsTrivial})). If $\xi \neq {\textbf 1}$, then from (\ref{EqXiIsNotTrivial}) we know that
$$
\bigl(1 - \chi(-1)\xi(-1)\bigr) \mathop{\sum_{n \leq q}}_{n \in \s(Q)} \frac{\overline{\chi}(n)}{n} \ll
(\log Q)^{2/3 + o(1)}
$$
where the constant is absolute and $o(1) \to 0$ as $q \to \infty$.

Putting all of this together with P\'{o}lya's Fourier expansion, we deduce the claimed bound on $S_\chi(t)$.
\end{proof}

\section{Multiplicative non-mimicry: proof of Theorem \ref{GeneralizeLemma32}}
\label{SectLowerBoundOnDistance}

In Lemma 3.2 of \cite{GSPretentiousCharactersAndPV}, Granville and Soundararajan proved that
for any primitive character $\chimodq$ of odd order $g$, and any primitive character $\xi$ of opposite parity and conductor smaller than a power of $\log y$,
\begin{equation}
\label{Lemma32}
\D(\chi, \xi; y)^2 \geq
\bigl(\delta_g + o(1)\bigr) \log \log y .
\end{equation}
Our goal in this section is to prove Theorem \ref{GeneralizeLemma32}, which asserts that the same lower bound continues to hold for small perturbations of $\xi$. To be precise, we will show that under the same hypotheses on $\chi$ and $\xi$ as above,
\begin{equation}
\label{BoundFromGenLemma32}
\D\bigl(\chi(n), \xi(n) n^{i \beta}; y\bigr)^2
\geq
\bigl(\delta_g + o(1)\bigr) \log \log y
\end{equation}
for all $\beta$ of magnitude smaller than $\log^2 y$. For
$\beta = o\!\left(\frac{ \log \log y}{\log y}\right)$ this is straightforward:
\begin{eqnarray*}
\mathbb{D}\left(\chi(n), \xi(n) n^{i\beta}; y\right)^2
&=&
\sum_{p \leq y} \frac{1}{p}
\left(1 - \text{Re } \chi \overline{\xi}(p) e^{-i\beta \log p}\right) \\
&=&
\sum_{p \leq y} \frac{1}{p}
\left(1 - \text{Re } \chi \overline{\xi}(p)
\left(1 + O\left(|\beta| \log p\right) \right) \right) \\
&=&
\mathbb{D}\left(\chi, \xi; y\right)^2 +
O\left(|\beta| \sum_{p \leq y} \frac{\log p}{p}\right) \\
&=& \mathbb{D}\left(\chi, \xi; y\right)^2 + o\left(\log \log
y\right)
\end{eqnarray*}
and thus for such $\beta$, (\ref{BoundFromGenLemma32}) follows from (\ref{Lemma32}).
For larger perturbations, however, the problem is more delicate.

Our plan of attack is as follows. Fix a primitive Dirichlet character $\chimodq$ of odd order $g$, and a primitive $\xi \, (\modulo m)$ of opposite parity to $\chi$. Since $\chi$ has odd order, $\chi(-1) = 1$, whence $\xi(-1) = -1$ and therefore $\xi$ has even order $k$, say. We partition the interval $[2,y]$ into many small intervals of the form $(x,(1+\delta)x]$, where $\delta$ is small. For each prime $p$ in such an interval, we approximate $p^{-i\beta}$ by $x^{-i\beta}$. This reduces our problem to estimating sums of the form
$$
\sum_{\ell \, (\text{mod } k)}
\mathop{\sum_{x < p \leq (1+\delta) x}}_{\xi(p) = e\left(\frac{\ell}{k}\right)}
\frac{1}{p} \left(1 - \text{Re } \chi(p) \, e\!\left(- \frac{\ell}{k}\right) x^{-i\beta} \right) .
$$
Following Granville and Soundararajan's proof of (\ref{Lemma32}), we ignore the arithmetic properties of $\chi$ and view it as an arbitrary function from $\Z$ to $\mu_g \cup \{0\}$; here $\mu_g$ denotes the set of $g^\text{th}$ roots of unity. This leads us to consider
$$
\sum_{\ell \, (\text{mod } k)}
\mathop{\sum_{x < p \leq (1+\delta) x}}_{\xi(p) = e\left(\frac{\ell}{k}\right)}
\frac{1}{p} \min_{z \in \mu_g \cup \{0\}} \left(1 - \text{Re } z \, e\!\left(- \frac{\ell}{k}\right) x^{-i\beta} \right) ,
$$
and since the only factor dependent on $p$ is the $\frac{1}{p}$ out front, we look at
$$
\mathop{\sum_{x < p \leq (1+\delta) x}}_{\xi(p) = e\left(\frac{\ell}{k}\right)}
\frac{1}{p}  .
$$
We expect $\xi(p) = e\big(\frac{\ell}{k}\big)$ for $1/k$ of the primes, so the natural guess is
\[
\mathop{\sum_{x < p \leq (1+\delta) x}}_{\xi(p) = e\left(\frac{\ell}{k}\right)} \frac{1}{p}
\approx
\frac{1}{k} \sum_{x < p \leq (1+\delta) x} \frac{1}{p}
\approx
\frac{\delta}{k \, \log x} .
\]
A straightforward application of Siegel-Walfisz will make this estimate rigorous (see Lemma \ref{SWapp}), and the remaining sum
$$
\sum_{\ell \, (\text{mod } k)}
\min_{z \in \mu_g \cup \{0\}} \left(1 - \text{Re } z \, e\!\left(- \frac{\ell}{k}\right) x^{-i\beta} \right)
$$
will then be evaluated by arguments inspired by those of \cite{GSPretentiousCharactersAndPV}.
Summing over all the small intervals will yield the desired lower bound (\ref{BoundFromGenLemma32}).

\subsection{The contribution from short intervals}


Our first goal is to obtain a lower bound on the sum over a short
interval
\begin{equation}
\label{interval}
\sum_{x < p \leq (1+\delta)x} \frac{1}{p} \left( 1
- \text{Re } \chi \overline{\xi} (p) \, p^{-i\beta} \right)
\end{equation}
where
$$
\delta \asymp \frac{1}{\log^3 y}  \, .
$$
Note that for any prime $p \in \bigl(x , (1+\delta)x\bigr]$, we may
approximate $p^{i \beta}$ by $x^{i \beta}$: we have
${0 \leq \log p - \log x \leq \delta}$,
whence
\begin{eqnarray*}
\left| p^{-i\beta} - x^{-i\beta} \right|
&=&
\left| 1 - e^{i\beta (\log p - \log x)} \right| \\
& \leq &
\bigl| \beta (\log p - \log x) \bigr| \\
& \leq &
\delta |\beta|  .
\end{eqnarray*}
Therefore,
\begin{eqnarray}
\sum_{x < p \leq (1 + \delta) x}
\frac{1}{p} \left(1 - \text{Re } \chi \overline{\xi}(p) p^{-i\beta}\right)
\!\!\!
&=&
\!\!\!\!\!\!\!
\sum_{x < p \leq (1 + \delta) x}
\frac{1}{p} \left(1 - \text{Re } \chi \overline{\xi}(p) x^{-i\beta}\right) +
O\left(\delta |\beta| \sum_{x < p \leq (1 + \delta) x} \frac{1}{p} \right) \nonumber \\
&=&
\!\!\!\!\!\!\!
\sum_{x < p \leq (1 + \delta) x} \frac{1}{p} \left(1 - \text{Re
} \chi \overline{\xi}(p) e(\theta_x) \right) + O\left(\frac{\delta^2
\log^2 y}{\log x} \right) \label{sect2bigsum}
\end{eqnarray}
where $\theta_x = -\frac{\beta}{2\pi} \log x$.
We bound the sum from below in terms of the orders of $\chi$ and
$\xi$:
\begin{eqnarray*}
\sum_{x < p \leq (1+\delta)x} \frac{1}{p} \left(1 - \text{Re } \chi \overline{\xi}(p) \, e(\theta_x)\right)
&=&
\!\!\!\!
\sum_{\ell \, (\text{mod } k)}
\mathop{\sum_{x < p \leq (1+\delta)x}}_{\xi(p) = e\left(\frac{\ell}{k}\right)}
\frac{1}{p} \Biggl(1 - \text{Re } \chi(p) \, e\biggl(- \frac{\ell}{k}\biggr) e(\theta_x)\Biggr) \\
& \geq &
\!\!\!\!
\sum_{\ell \, (\text{mod } k)} \mathop{\sum_{x < p \leq
(1+\delta)x}}_{\xi(p) = e\left(\frac{\ell}{k}\right)} \frac{1}{p}
\min_{z \in \mu_g \cup \{0\}} \Biggl(1 - \text{Re } z \cdot
e\biggl(\theta_x - \frac{\ell}{k}\biggr) \Biggr)
\end{eqnarray*}
We first estimate the interior sum over primes:
\begin{lemma}
\label{SWapp}
Suppose $\epsilon > 0$, $\xi \, (\text{mod } m)$ is a nonprincipal character
of order $k$, and ${y \geq \exp(m^\epsilon)}$. Then for $\delta \asymp (\log y)^{-3}$ and
${x \geq \exp\bigl((\log y)^\epsilon\bigr)}$,
$$
\mathop{\sum_{x < p \leq (1+\delta)x}}_{\xi(p) =
e\left(\frac{\ell}{k}\right)} \frac{1}{p} = \frac{\delta}{k \log x}
\bigl(1+o(1)\bigr)
$$
where $o(1) \to 0$ as $y \to \infty$ and depends only on $y$ and
$\epsilon$.
\end{lemma}
Note that this estimate is independent of $\ell$. Thus, the
following general result, combined with Lemma \ref{SWapp}, will
furnish a lower bound on the sum (\ref{interval}):
\begin{lemma}
\label{summin} Given $g \geq 3$ odd, $k \geq 2$ even, and
$\theta \in \left(-\frac{1}{2}, \frac{1}{2}\right]$.
Set $k^* = \frac{k}{(g,k)}$. Then
\begin{equation}
\label{Flem}
\frac{1}{k} \sum_{\ell \, (\text{mod } k)} \min_{z \in \mu_g \cup\{0\}}
\Biggl( 1 - \text{Re } z \cdot e\biggl(\theta - \frac{\ell}{k}\biggr) \Biggr)
\, = \,
1 - \frac{\sin \frac{\pi}{g}}{k^* \, \tan \frac{\pi}{gk^*}} F_{gk^*}\!\left(-gk^*\theta\right)
\end{equation}
where
$$
F_N(\omega) = \cos \frac{2\pi\{\omega\}}{N} + \left(\tan
\frac{\pi}{N}\right) \sin \frac{2\pi\{\omega\}}{N} .
$$
\end{lemma}
To make sense of this lemma, we examine some properties of
$F_N(\omega)$. First, since $F_N(\omega) = F_N(\{\omega\})$ we may
assume that $\omega \in [0,1)$. Second, since $k^*$ must be even,
$gk^* \geq 6$, and we can therefore assume that $N \geq 6$. Under
these assumptions, one easily checks that
\begin{description}
\item[({\em i})] $F_N(0)=1$ and $F_N(0.5) = \frac{1}{\cos \frac{\pi}{N}}$, \label{FnPropPage}

\item[({\em ii})] $F_N(\omega)$ is concave down everywhere on $[0,1)$,

\item[({\em iii})] On the unit interval, $F_N$ is symmetric about $\omega = \frac{1}{2}$, and

\item[({\em iv})] The average value of $F_N$ over the unit interval is $\frac{N}{\pi} \tan \frac{\pi}{N}$.
\end{description}
Thus, for the `typical' $\theta$ we expect the right
side of (\ref{Flem}) to be $\delta_g$. It is appreciably larger
than $\delta_g$ when $gk^*\theta$ is close to an integer, and
somewhat smaller than $\delta_g$ when $gk^*\theta$ is close to
a half-integer. In the context of \cite{GSPretentiousCharactersAndPV}, $\theta = 0$, which
allowed Granville and Soundararajan to bound (\ref{Flem}) from below
by $\delta_g$ quite easily. Although our arguments are also not
difficult, the computations are naturally somewhat more involved; we
will isolate the proof in a separate subsection.

Before proving the two lemmata, we deduce from them a lower bound
on (\ref{interval}). The main term of (\ref{sect2bigsum}) can be bounded from
below as follows, for all $x \geq \exp\bigl((\log y)^\epsilon\bigr)$:
\begin{eqnarray*}
\sum_{x < p \leq (1 + \delta) x} \frac{1}{p} \left(1 - \text{Re } \chi \overline{\xi}(p) e(\theta_x) \right)
\!\!
& \geq &
\!\!\!\!\!\!
\sum_{\ell \, (\text{mod } k)}
\left(
\mathop{\sum_{x < p \leq (1+\delta)x}}_{\xi(p) = e\left(\frac{\ell}{k}\right)} \frac{1}{p}
\right)
\min_{z \in \mu_g \cup \{0\}} \Biggl(1 - \text{Re } z \cdot
e\biggl(\theta_x - \frac{\ell}{k}\biggr) \Biggr) \\
& = &
\frac{\delta \bigl( 1 + o(1) \bigr)}{\log x}
\left(
1 - \frac{\sin \frac{\pi}{g}}{k^* \tan \frac{\pi}{gk^*}} \, F_{gk^*}(-gk^* \theta_x)
\right)
\end{eqnarray*}
Let
$$
G(t) = 1 - \frac{\sin \frac{\pi}{g}}{k^* \tan \frac{\pi}{gk^*}} \, F_{gk^*}\left(\frac{\beta g k^*}{2\pi} \, t\right) .
$$
Note that $G$ is minimized at values of $t$ for which $F_{gk^*}$ is maximized, whence $G(t) \geq 1 - \frac{\sin \frac{\pi}{g}}{k^* \sin \frac{\pi}{g k^*}}$. It follows that as a function of $t$, $G(t)$ is bounded away from 0. This combined with our choice of $\delta$ of size $(\log y)^{-3}$ shows that we can bound (\ref{sect2bigsum}) as follows:
\begin{eqnarray}
\sum_{x < p \leq (1 + \delta) x} \frac{1}{p} \left(1 - \text{Re } \chi \overline{\xi}(p) p^{-i\beta}\right)
& \geq &
\frac{\bigl(1 + o(1) \bigr) \delta }{\log x} \, G(\log x) + O\left(\frac{\delta^2 \log^2 y}{\log x}\right) \nonumber \\
& = &
\frac{\bigl( 1 + o(1) \bigr) \delta }{\log x} \, G(\log x)  \label{lowerboundshortint}
\end{eqnarray}
where the $o(1)$ term in (\ref{lowerboundshortint}) tends to 0 as $y \to \infty$ and
depends only on $y$, $\epsilon$, $g$, and $k$.

We now go back and prove the two lemmata.

\begin{proof}[Proof of Lemma \ref{SWapp}]
A consequence of the Siegel-Walfisz Theorem says that for any fixed
$\epsilon>0$ and $A>0$, for all $X \geq \exp \left(m^\epsilon\right)$,
$$
\theta(X; m, a) := \!\!\!\!\!\!
\mathop{\sum_{p \leq X}}_{p \, \equiv \, a \, (\modulo m)}
\!\!\!\!\!\! \log p
=
\frac{X}{\varphi(m)} \Biggl( 1 + O\biggl(\frac{1}{(\log X)^{A}}\biggr) \Biggr)
$$
where the constant implicit in the $O$-term depends only upon $A$ and $\epsilon$.
In particular, for all $X \geq \exp \bigl((\log y)^\epsilon\bigr)$,
\begin{equation}
\label{ourSW}
\theta(X; m, a)
=
\frac{X}{\varphi(m)} \Biggl( 1 + O_\epsilon\biggl(\frac{1}{(\log X)^{4 / \epsilon}}\biggr) \Biggr)
\end{equation}
where the implicit constant only depends on $\epsilon$.

To apply Siegel-Walfisz, we must first express the sum in question
as a sum over primes in arithmetic progressions:
$$
\mathop{\sum_{x < p \leq (1+\delta)x}}_{\xi(p) =
e\left(\frac{\ell}{k}\right)} \frac{1}{p} = \mathop{\sum_{a \,
(\modulo m)}}_{\xi(a) = e\left(\frac{\ell}{k}\right)}
\mathop{\sum_{x < p \leq (1+\delta)x}}_{p \, \equiv \, a \, (\modulo
m)} \frac{1}{p} .
$$
Note that ${x < p \leq (1+\delta)x}$ is equivalent to ${\frac{1}{1+\delta} \, p \leq x
< p}$, whence
$$
\frac{x\log x}{p\log p} = \frac{x}{p} \cdot \frac{\log x}{\log p} =
\bigl(1 + O(\delta)\bigr) \cdot \Biggl(1 +
O\biggl(\frac{\delta}{\log p}\biggr)\Biggr) = 1 + O(\delta) .
$$
Combining this with (\ref{ourSW}) and the hypotheses on the sizes of $x$ and $\delta$ yields
\begin{eqnarray}
\mathop{\sum_{x < p \leq (1+\delta)x}}_{p \, \equiv \, a \, (\modulo m)} \frac{1}{p}
&=&
\frac{1+ O(\delta)}{x \log x}
\mathop{\sum_{x < p \leq (1+\delta)x}}_{p \, \equiv \, a \, (\modulo m)} \!\!\!\!\!\! \log p \nonumber \\
&=&
\label{innersum}
\frac{\delta}{\varphi(m) \, \log x} \Biggl( 1 + O_\epsilon \left(\frac{1}{\log y}\right) \Biggr) .
\end{eqnarray}
Since this estimate is independent of $a$, to prove the lemma it remains only to show that
\begin{equation}
\label{xiequidist} \mathop{\sum_{a \, (\modulo m)}}_{\xi(a) \, = \,
e\left(\frac{\ell}{k}\right)} \!\!\! 1 = \frac{\varphi(m)}{k} .
\end{equation}
For brevity, denote $\left(\Z / m\Z\right)^*$ by $G$. Since $\xi$
has order $k$, there is some $b \in G$ such that $1, \xi(b),
\xi(b)^2, \ldots, \xi(b)^{k-1}$ are all distinct; on the other hand,
all these must be $k^\text{th}$ roots of unity. In particular, there
exists some $g \in G$ such that $\xi(g) =
e\left(\frac{1}{k}\right)$.

Let $H$ be the kernel of $\xi$, i.e. $H = \{ a \in G : \xi(a) = 1\}$. This is a
normal subgroup of $G$, and $g^\ell H = \left\{ a \in G : \xi(a) =
e\left(\frac{\ell}{k}\right) \right\}$. $G$ can therefore be
decomposed as a disjoint union of the $k$ cosets $g^\ell H$ with $0
\leq \ell \leq k-1$. Since $\left|g^\ell H \right| = \left| H
\right|$, (\ref{xiequidist}) must hold. Combining this with
(\ref{innersum}) yields the lemma.
\end{proof}

\subsubsection{Proof of Lemma \ref{summin}}

Recall that $g \geq 3$ is odd, $k \geq 2$ is even, and $\theta \in
\left(-\frac{1}{2}, \frac{1}{2}\right]$. Let $d = (g,k)$ and set
$k^* = \frac{k}{d}$ and $g^* = \frac{g}{d}$.

To prove (\ref{Flem}), it suffices to show
\begin{equation}
\label{two} \sum_{\ell \, (\text{mod } k)} \max_{z \in \mu_g \cup
\{0\}} \text{Re } z \cdot e\biggl(\theta - \frac{\ell}{k}\biggr) = d
\cdot \frac{\sin \frac{\pi}{g}}{\tan \frac{\pi}{gk^*}} \cdot
F_{gk^*}\!\left(-gk^*\theta\right)
\end{equation}
Let $\mathcal{A}_{{}_0} = \left\{e(\beta) : -\frac{1}{2g} < \beta
\leq \frac{1}{2g} \right\}$ and set $\mathcal{A}_n =
e\!\left(\frac{n}{g}\right) \mathcal{A}_{{}_0}$; note that the
disjoint union of $\mathcal{A}_n$ as $n$ runs over any complete set
of residues of $\Z / g\Z$ is the complex unit circle. In particular,
for any $\ell \in \Z$ there is a unique $n_{{}_\ell} \in
\left(-\frac{g}{2}, \frac{g}{2}\right]$ such that $e\!\left(\theta -
\frac{\ell}{k}\right) \in \mathcal{A}_{n_{{}_\ell}}$. By definition,
this means that $e\!\left(-\frac{n_{{}_\ell}}{g}\right)
e\!\left(\theta - \frac{\ell}{k}\right) \in \mathcal{A}_{{}_0}$.
Since for all other $n \in \left(-\frac{g}{2}, \frac{g}{2}\right]$
we have $e\!\left(-\frac{n}{g}\right) e\!\left(\theta -
\frac{\ell}{k}\right) \not\in \mathcal{A}_{{}_0}$, we deduce that
\begin{eqnarray*}
\max_{z \in \mu_g \cup \{0\}} \text{Re } z \cdot e\!\left(\theta - \frac{\ell}{k}\right)
&=&
\text{Re } e\!\left(- \frac{n_{{}_\ell}}{g}\right) \, e\!\left(\theta - \frac{\ell}{k}\right) \\
&=&
\text{Re } e(\theta) \, e\!\left(\frac{f(\ell)}{gk}\right)
\end{eqnarray*}
where $f : \Z \to \Z$ is defined $f(\ell) = -(g\,\ell + k
\,n_{{}_\ell})$. This allows us to rewrite the left hand side of the
inequality (\ref{two}):
\begin{equation}
\label{three}
\sum_{\ell \, (\text{mod } k)} \max_{z \in \mu_g \cup \{0\}}
\text{Re } z \cdot e\biggl(\theta - \frac{\ell}{k}\biggr) =
\text{Re } e(\theta) \sum_{\ell \, (\text{mod } k)} e\left(\frac{f(\ell)}{gk}\right) .
\end{equation}
Our aim is rewrite the sum on the right side of (\ref{three}) in
terms of geometric series.

It is not hard to see that if $\ell_1 \equiv \ell_2 \, (\text{mod }
k)$ then $f(\ell_1) \equiv f(\ell_2) \, (\text{mod } gk)$. However,
more is true:
\begin{lemma}
\label{lem1}
$\ell_1 \equiv \ell_2 \, (\text{mod } k^*) \Longrightarrow f(\ell_1) \equiv f(\ell_2) \, (\text{mod } gk)$
\end{lemma}
\begin{proof}
Given $\ell_1 \equiv \ell_2 \, (\text{mod } k^*)$. Then $k \mid g\,
(\ell_2-\ell_1)$, since $g(\ell_2 - \ell_1) = g^* k \, \frac{\ell_2
- \ell_1}{k^*}$. Equivalently, there exists $m \in \Z$ such that
$-\frac{\ell_1}{k} = -\frac{\ell_2}{k} + \frac{m}{g}$. Therefore, by
the definition of $n_{{}_\ell}$, we find that both $e\!\left(\frac{m
- n_{{}_{\ell_1}}}{g}\right)$ and
$e\!\left(-\frac{n_{{}_{\ell_2}}}{g}\right)$ belong to the set
$e\!\left(\frac{{\ell_2}}{k} - \theta \right) \mathcal{A}_{{}_0}$.
But this implies that ${n_{{}_{\ell_1}} \equiv m + n_{{}_{\ell_2}} \,
(\text{mod } g)}$, whence
$$
e\!\left(\frac{f(\ell_1)}{gk}\right) = e\!\left(\frac{f(\ell_2)}{gk}\right)
$$
and we conclude.
\end{proof}

Thus, we can restrict the sum on the right side of (\ref{three}) to
$\Z / k^* \Z$:
\begin{equation}
\label{four} \sum_{\ell \, (\text{mod } k)}
e\!\left(\frac{f(\ell)}{gk}\right) = d \cdot \!\!\!\!\! \sum_{\ell^*
\, (\text{mod } k^*)} e\!\left(\frac{f(\ell^*)}{gk}\right) .
\end{equation}

We now prove a weaker form of Lemma \ref{lem1}, which has the advantage of a converse.
\begin{lemma}
\label{lem2}
$\ell_1 \equiv \ell_2 \, (\text{mod } k^*) \Longleftrightarrow f(\ell_1) \equiv f(\ell_2) \, (\text{mod } k)$
\end{lemma}
\begin{proof}
\begin{eqnarray*}
f(\ell_1) \equiv f(\ell_2) \, (\text{mod } k)
& \Longrightarrow & k \mid g(\ell_2 - \ell_1) \\
& \Longrightarrow & k^* \mid g^*(\ell_2 - \ell_1) \\
& \Longrightarrow & \ell_1 \equiv \ell_2 \, (\text{mod } k^*)
\end{eqnarray*}
since $(g^*, k^*) = 1$. On the other hand,
$$
\ell_1 \equiv \ell_2 \, (\text{mod } k^*) \Longrightarrow k \mid d(\ell_2 - \ell_1)
$$
whence
\begin{eqnarray*}
f(\ell_1) - f(\ell_2) &=& g(\ell_2 - \ell_1) + k(n_{\ell_1} - n_{\ell_2}) \\
&=& g^*d(\ell_2 - \ell_1) + k(n_{\ell_1} - n_{\ell_2}) \\
& \equiv & 0 \, (\text{mod } k).
\end{eqnarray*}
\end{proof}

\begin{prop}
\label{prop3} The map $f$ restricted to $\left[-\frac{k^*}{2} + k^*
\theta, \frac{k^*}{2} + k^* \theta\right) \cap \Z$ is an injection
into
$$
\left(-\frac{k}{2} - g k \theta, \frac{k}{2} - g k \theta\right] \cap \Z .
$$
\end{prop}
\begin{proof}
Injectivity follows immediately from Lemma \ref{lem2}, so it
suffices to show that the image of $\left[-\frac{k^*}{2} + k^*
\theta, \frac{k^*}{2} + k^* \theta\right) \cap \Z$ under $f$ lands
in the claimed target. In fact, we will show a slightly stronger
statement. Observe that because $|\theta| \leq \frac{1}{2}$,
$$
\left[-\frac{k^*}{2} + k^* \theta, \frac{k^*}{2} + k^* \theta\right)
\subseteq
\left[-\frac{k}{2} + k \theta, \frac{k}{2} + k \theta\right) ;
$$
we claim that the image under $f$ of the larger
set lands inside the claimed target.

Fix any $\ell \in \left[-\frac{k}{2} + k \theta, \frac{k}{2} + k
\theta\right)$; this is equivalent to requiring $\theta -
\frac{\ell}{k} \in \left(-\frac{1}{2}, \frac{1}{2}\right]$. By
definition of $n_{{}_\ell}$ we have $e\!\left(\theta -
\frac{\ell}{k}\right) \in \mathcal{A}_{n_{{}_\ell}}$, from which we
deduce that for some integer $N$,
$$
\theta - \frac{\ell}{k} \in
\left(N + \frac{2 n_{{}_\ell} - 1}{2g}, N + \frac{2 n_{{}_\ell} +
1}{2g}\right] .
$$
By our restriction on $\ell$, $N$ must equal 0
(recall that $-\frac{g-1}{2} \leq n_{{}_\ell} \leq \frac{g-1}{2}$).
It follows that $f(\ell) \in \left(-\frac{k}{2} - g k \theta,
\frac{k}{2} - g k \theta\right]$.
\end{proof}

Note that $d \, | \, f(\ell)$ for all $\ell$. Combining this fact with Proposition \ref{prop3} we conclude that
$$
\left\{f(\ell^*) : -\frac{k^*}{2} + k^* \theta \leq \ell^* < \frac{k^*}{2} + k^* \theta \right\}
$$
is a set of $k^*$ distinct multiples of $d$, all contained in
$\left(-\frac{k}{2} - g k \theta, \frac{k}{2} - g k \theta\right]$.
But by inspection, this interval contains \textit{precisely} $k^*$ multiples of $d$. Therefore:
\begin{eqnarray}
\sum_{\ell^* \, (\text{mod } k^*)} e\!\left(\frac{f(\ell^*)}{gk}\right)
& = &
\sum_{-\frac{k^*}{2} + k^* \theta \leq \ell^* < \frac{k^*}{2} + k^* \theta}
e\!\left(\frac{f(\ell^*)}{gk}\right) \nonumber \\
& = &
\sum_{\frac{1}{d}\left(-\frac{k}{2} - g k \theta\right) < m \leq \frac{1}{d}\left(\frac{k}{2} - g k \theta\right)} e\!\left(\frac{md}{gk}\right) \nonumber \\
& = &
\sum_{-\frac{k^*}{2} - g k^* \theta < m \leq \frac{k^*}{2} - g k^* \theta}
e\!\left(\frac{m}{gk^*}\right) \label{five}
\end{eqnarray}
This is a $k^*$-term geometric series with first term
$e\!\left(\frac{1}{gk^*}\left[\frac{k^*}{2} - gk^*\theta\right]\right)$ and ratio
$e\!\left(-\frac{1}{gk^*}\right)$.
Summing the series and performing standard algebraic manipulations, one finds
$$
\sum_{-\frac{k^*}{2} - g k^* \theta < m \leq \frac{k^*}{2} - g k^* \theta}
e\!\left(\frac{m}{gk^*}\right) =
e\!\left(-\theta + \frac{1 - 2c}{2gk^*}\right) \, \frac{\sin \frac{\pi}{g}}{\sin \frac{\pi}{gk^*}}
$$
where $c = \left\{- gk^*\theta\right\} \in [0,1)$. Tracing back
through equations (\ref{three})-(\ref{five}) and simplifying, we see
that
\begin{eqnarray*}
\sum_{\ell \, (\text{mod } k)} \max_{z \in \mu_g \cup \{0\}}
\text{Re } z \cdot e\biggl(\theta - \frac{\ell}{k}\biggr)
& = &
d \cdot \frac{\sin \frac{\pi}{g}}{\sin \frac{\pi}{gk^*}} \cdot \cos \left(\frac{\pi}{gk^*} (1-2c)\right) \\
& = & \label{six} d \cdot \frac{\sin \frac{\pi}{g}}{\tan
\frac{\pi}{gk^*}} \cdot F_{gk^*}\bigl(-gk^*\theta\bigr)
\end{eqnarray*}
proving (\ref{two}), and thus the lemma. \hfill $\qed$

\subsection{Completing the proof of Theorem \ref{GeneralizeLemma32}}

Let $x_0 = \exp\bigl((\log y)^\epsilon\bigr)$ and set $x_r = x_0 (1 + \delta)^r$. Then from (\ref{lowerboundshortint}) we deduce
\begin{eqnarray}
\D(\chi(n), \xi(n) \, n^{i\beta}; y)^2
&=&
\sum_{p \leq y} \frac{1}{p} \left(1 - \text{Re } \chi \overline{\xi}(p) p^{-i\beta}\right)
\nonumber \\
& \geq &
\sum_{x_0 < p \leq y} \frac{1}{p} \left(1 - \text{Re } \chi \overline{\xi}(p) p^{-i\beta}\right) \nonumber \\
& \geq &
\mathop{\sum_{r \geq 0}}_{x_{r+1} \leq y} \sum_{x_r < p \leq x_{r+1}}
\frac{1}{p} \left(1 - \text{Re } \chi \overline{\xi}(p) p^{-i\beta}\right) \nonumber \\
& \geq &
\mathop{\sum_{r \geq 0}}_{x_{r+1} \leq y} \frac{ \bigl( 1 + o(1) \bigr) \delta }{\log x_r} \, G(\log x_r) \nonumber \\
& \geq &
\bigl( 1 + o(1) \bigr)
\log (1+\delta) \mathop{\sum_{r \geq 0}}_{x_{r+1} \leq y} \frac{G(\log x_r) }{\log x_r}
\label{longchain}
\end{eqnarray}
We recognize the sum above as the left Riemann sum -- with subintervals of length ${\log (1+\delta)}$ -- for the integral
$\displaystyle \int_{\log x_0}^{\log x_m} \frac{G(t)}{t} \dx t$, where $m$ is the integer such that ${x_m \leq y < x_{m+1}}$.
Since
\begin{eqnarray*}
\left| \frac{d}{dt} \left(\frac{G(t)}{t}\right) \right|
& \leq &
\left|\frac{G'(t)}{t}\right| + \left| \frac{G(t)}{t^2} \right| \\
& \leq &
\frac{\frac{\sin \frac{\pi}{g}}{k^* \tan \frac{\pi}{gk^*}} \, F'_{gk^*}(0)}{\log x_0}
+ \frac{2}{(\log x_0)^2} \\
& \ll &
1
\end{eqnarray*}
for all $t \geq \log x_0$, we have
\begin{eqnarray*}
\left| \log (1+\delta) \mathop{\sum_{r \geq 0}}_{x_{r+1} \leq y} \frac{G(\log x_r) }{\log x_r}
- \int_{\log x_0}^{\log y} \frac{G(t)}{t} \, \dx t \right|
& \ll &
(\log y) \cdot \log (1+\delta)
+ \left| \int_{\log x_m}^{\log y} \frac{G(t)}{t} \, \dx t \right| \\
& \ll &
\frac{1}{\log^2 y}
\end{eqnarray*}
Therefore, continuing our calculation from where we left it in (\ref{longchain}),
\begin{equation}
\label{integral}
\D(\chi(n), \xi(n) \, n^{i\beta}; y)^2
\geq
\bigl( 1 + o(1) \bigr) \int_{\log x_0}^{\log y} \frac{G(t)}{t} \, \dx t + O(1) .
\end{equation}

To prove Theorem \ref{GeneralizeLemma32} it remains only to bound the integral on the right side of (\ref{integral}) from below by $\bigl(\delta_g + o(1) \bigr) \log \log y$.
Recall that
$$
G(t) =
1 - \frac{\sin \frac{\pi}{g}}{k^* \tan \frac{\pi}{gk^*}} \, F_{gk^*}\left(\frac{\beta g k^*}{2\pi} \, t\right)
$$
where $F_N(\omega) = \cos \frac{2\pi\{\omega\}}{N} + \left(\tan
\frac{\pi}{N}\right) \sin \frac{2\pi\{\omega\}}{N}$ is concave down everywhere on the unit interval and symmetric about $t = \frac{1}{2}$, with minima at the endpoints of the interval. Furthermore, $\overline{F_N}$, the mean value of $F_N$ on the unit interval, is $\frac{N}{\pi} \tan \frac{\pi}{N}$. Rewriting (\ref{integral}), we see that it suffices to prove that
$$
\int_{a(y)}^{b(y)} \frac{1}{t} F_{N}(t) \dx t
\leq
\bigl(\overline{F_{N}}+o(1)\bigr) \log \log y
$$
where $a(y) = \frac{N|\beta|}{2\pi} \, (\log y)^\epsilon$ and $b(y) = \frac{N|\beta|}{2\pi} \, \log y$.
(Note that $a(y)$ and $b(y)$ are expressed in terms of the magnitude of $\beta$, a change of variables we can make because $F_N$ is an even function.)
Given any $x \geq 1$ we find
$$
\int_1^x \frac{1}{t} \, F_N(t) \dx t = \overline{F_N} \cdot \log x + O(1)  ,
$$
by splitting the integral into unit intervals (with at most one exception) and on each interval bounding $\frac{1}{t}$ from above and below trivially. Thus if $a(y) \geq 1$, we immediately find
\begin{eqnarray*}
\int_{a(y)}^{b(y)} \frac{1}{t} \, F_N(t) \dx t
&=&
\overline{F_N} \cdot \log \frac{b(y)}{a(y)} + O(1) \\
& \leq &
\Bigl(\overline{F_N} + o(1) \Bigr) \, \log\log y  .
\end{eqnarray*}

Now we consider the case when $a(y) < 1$. Note that we may take $b(y) \geq 1$: from the discussion directly following equation (\ref{BoundFromGenLemma32}) we see that we can assume $\displaystyle |\beta| \geq \frac{C_0 (\log \log y)^{1/2}}{\log y}$ for any positive constant $C_0$, and since $y \geq 3$ and $N = g k^* \geq 6$, choosing $\displaystyle C_0 = \frac{2 \pi}{6} \, (\log \log 3)^{-1/2}$ makes $b(y) \geq 1$.
Therefore, \begin{eqnarray*}
\int_{a(y)}^{b(y)} \frac{1}{t} \, F_N(t) \dx t
&=&
\int_{a(y)}^{1} \frac{1}{t} \, F_N(t) \dx t + \int_{1}^{b(y)} \frac{1}{t} \, F_N(t) \dx t \\
&=&
\int_1^{\frac{1}{a(y)}} \frac{1}{t} \, F_N\left(\frac{1}{t}\right) \dx t + \overline{F_N} \cdot \log b(y) + O(1)
\end{eqnarray*}
It remains only to show that
\begin{equation}
\label{suffices}
\int_1^x \frac{1}{t} \, F_N\left(\frac{1}{t}\right) \dx t \leq
\overline{F_N} \cdot \log x + O(1) .
\end{equation}
Because $F_N$ is concave down on $[0,1)$, we see that for all sufficiently large $x$, $F_N\Bigl(\frac{1}{x}\Bigr) \leq \overline{F_N}$. Therefore,
$$
\frac{d}{\dx x}\Biggl(\int_1^x \frac{1}{t} \, F_N\left(\frac{1}{t}\right) \dx t\Biggr)
\leq
\frac{d}{\dx x}\left(\overline{F_N} \cdot \log x\right)
$$
for all large $x$. This implies (\ref{suffices}), and Theorem \ref{GeneralizeLemma32} is proved.
\hfill $\qed$

\section{Proof of Theorem 3}
\label{SectCharLowerBd}

All results stated and proved in this section are conditional on the Generalized Riemann Hypothesis.

In Theorem 2 we proved that
\[
| S_\chi(t) | \ll_g \sqrt{q} (\log \log q)^{1 - \delta_g + o(1)}
\]
for any primitive character $\chimodq$ of odd order $g \geq 3$. The goal of this section is to construct an infinite family of characters $\chimodq$ of order $g$ such that
\[
\max_{t \leq q}| S_\chi(t) | \gg_{\epsilon,g} \sqrt{q} (\log \log q)^{1 - \delta_g - \epsilon}
\]
thus showing that the constant $1-\delta_g$ in our upper bound cannot be improved. We note that when $g$ is squarefree, the dependence of the implicit constant on $g$ can be made explicit from our construction.

We first quote a result of Granville and Soundararajan:
\begin{theorem}[see Theorem 2.5 of \cite{GSPretentiousCharactersAndPV}]
\label{Thm:GS25}
Assume the GRH. Given a primitive character $\chimodq$, let ${\xi \, (\text{mod } m)}$ be a primitive character of opposite parity to $\chi$. Then
\[
\max_{t \leq q} |S_\chi(t)|  + \frac{\sqrt{m}}{\varphi(m)} \sqrt{q} \log \log \log q
\gg
\frac{\sqrt{m}}{\varphi(m)} \sqrt{q} (\log \log q) e^{-\D(\chi,\, \xi;\, \log q)^2}
\]
\end{theorem}
To prove Theorem 3 it therefore suffices to show that there is an odd character
${\xi \, (\text{mod } m)}$ and an infinite family of characters $\chimodq$ of odd order $g$ such that
\[
\D(\chi, \xi; \log q)^2 \leq (\delta_g + \epsilon) \log \log \log q
\]
or equivalently, that
\begin{equation}
\label{Eq:UpperBdOnDist}
\sum_{p \leq \log q} \frac{1}{p} \, \text{Re } \chi(p) \overline{\xi(p)} \geq
\bigl( 1 - \delta_g - \epsilon \bigr) \log \log \log q .
\end{equation}
We will accomplish this in two steps. First, using ideas similar to that of the previous section, we will prove:
\begin{prop}
\label{Prop:LowerBdWithXi}
For any $\epsilon > 0$, there exists an odd character $\xi \, (\text{mod } m)$ such that for
$y \geq \exp(m^\epsilon)$,
\begin{equation}
\label{Eq:LowerBdWithoutChi}
\sum_{p \leq y} \frac{1}{p} \, \max_{z \in \mu_g \cup \{0\}} \text{\emph{Re} } z \, \overline{\xi(p)}
\geq
\big(1 - \epsilon + o(1)\big) (1 - \delta_g) \log \log y;
\end{equation}
$o(1) \to 0$ as $y \to \infty$.
\end{prop}
\vspace{-0.2in}
Given such a $\xi$, to deduce (\ref{Eq:UpperBdOnDist}) it suffices to find a $\chimodq$ whose values at primes up to $\log q$ coincide with the $z$ which maximize each term of (\ref{Eq:LowerBdWithoutChi}). Using a generalization of Eisenstein's reciprocity law and the Chinese Remainder Theorem, we will prove:
\begin{prop}
\label{Prop:ChiApproxFn}
Fix an odd integer $g \geq 3$, and let $\psi : \Z \longrightarrow \mu_g \cup \{0\}$ be a completely multiplicative function. Then there exists a constant $C = C(g) > 0$ and infinitely many Dirichlet characters $\chimodq$ of order $g$ such that $\chi(n) = \psi(n)$ for all $n \leq C \, \log q$ which are coprime to $g$.
\end{prop}
\vspace{-0.2in}
With these results in hand, Theorem 3 follows easily:
\vspace{-0.2in}
\begin{proof}[Proof of Theorem 3]
Proposition \ref{Prop:LowerBdWithXi} furnishes a character $\xi$ such that (\ref{Eq:LowerBdWithoutChi}) holds for all $y \geq \exp(m^\epsilon)$. For any such $y$, choose $z_p \in \mu_g \cup \{0\}$ so that
\[
\sum_{p \leq y} \frac{1}{p} \, \max_{z \in \mu_g \cup \{0\}} \text{Re } z \, \overline{\xi(p)}
=
\sum_{p \leq y} \frac{1}{p} \, \text{Re } z_p \, \overline{\xi(p)} .
\]
By Proposition \ref{Prop:ChiApproxFn} we can find infinitely many characters $\chimodq$ such that $\chi(p) = z_p$ for all $p \leq C \log q$ which are coprime to $g$. For any such $\chi$, we therefore have
\[
\sum_{p \leq C \log q} \frac{1}{p} \, \text{Re } \chi(p) \overline{\xi(p)}
=
\sum_{p \leq C \log q} \frac{1}{p} \, \text{Re } z_p \, \overline{\xi(p)} +
O\left(\sum_{p \mid g} \frac{1}{p}\right) .
\]
Since $g$ is fixed, (\ref{Eq:LowerBdWithoutChi}) implies (\ref{Eq:UpperBdOnDist}); applying Theorem \ref{Thm:GS25}
yields Theorem 3.
\end{proof}
\vspace{-0.2in}
It remains only to prove the two propositions.
\begin{proof}[Proof of Proposition \ref{Prop:LowerBdWithXi}]
Let $\xi \, (\text{mod } m)$ be an odd character. Then its order $k$ must be even, and (exactly as in the previous section) we have
\[
\sum_{p \leq y} \frac{1}{p}\, \max_{z \in \mu_g \cup \{0\}} \text{Re } z \overline{\xi(p)} =
\sum_{\ell \, (\text{mod } k)} \max_{z \in \mu_g \cup \{0\}} \text{Re } z e\Big(- \frac{\ell}{k}\Big)
\mathop{\sum_{p \leq y}}_{\xi(p) = e\big(\frac{\ell}{k}\big)} \frac{1}{p} .
\]
Siegel-Walfisz implies that
\[
\mathop{\sum_{p \leq y}}_{\xi(p) = e\big(\frac{\ell}{k}\big)} \frac{1}{p} =
\frac{1+o(1)}{k} \log \log y
\]
and relation (\ref{two}) (with $\theta = 0$) gives
\[
\sum_{\ell \, (\text{mod } k)} \max_{z \in \mu_g \cup \{0\}} \text{Re } z e\Big(- \frac{\ell}{k}\Big)
=
(g,k) \frac{\sin \frac{\pi}{g}}{\tan \frac{\pi}{gk^*}} .
\]
Putting these estimates together yields
\[
\sum_{p \leq y} \frac{1}{p}\, \max_{z \in \mu_g \cup \{0\}} \text{Re } z \, \overline{\xi(p)} =
\big(1 - \delta_g + o(1)\big) \frac{\frac{\pi}{gk^*}}{\tan \frac{\pi}{gk^*}} \log \log y .
\]
The function $\frac{x}{\tan x}$ tends to 1 from below as $x \to 0$, so to prove the proposition it suffices to find a sequence of $k^*$ tending to infinity. Since $g$ is fixed and $k^* = k / (g,k)$, this is easily achieved by choosing $\xi$ of order $k$ relatively prime to $g$.
\end{proof}

\begin{proof}[Proof of Proposition \ref{Prop:ChiApproxFn}]
Let $y$ be large (this is an auxiliary parameter which will tend to infinity).
Given a prime $p \dnd g$, there exists an integer $Q_p$ such that
$\displaystyle \Bigl( \frac{Q_p}{p} \Bigr)_g = \psi(p)$, where $\left( \frac{\cdot}{\cdot} \right)_g$ is the $g^\text{th}$ order residue symbol. By the Chinese Remainder Theorem, there exists a $Q = Q(y)$ satisfying
\begin{enumerate}
\item $Q \equiv Q_p \, (\text{mod } p)$ for all primes $p \leq y$ such that $p \dnd g$;
\item $Q \equiv 1 \, (\text{mod } g)$; and
\item $\displaystyle g \mathop{\prod_{p \leq y}}_{p \dnd g} p < Q \leq 2g \mathop{\prod_{p \leq y}}_{p \dnd g} p$.
\end{enumerate}
It follows that
\begin{equation}
\label{Eq:CRT}
\displaystyle \Bigl( \frac{Q}{p} \Bigr)_g = \psi(p)
\end{equation}
for all $p \leq y$ coprime to $g$.

We now wish to use reciprocity for the $g^\text{th}$ order residue symbol to obtain a $g^\text{th}$-order character of modulus $Q$. For $g$ an odd prime, this is given by the Eisenstein reciprocity law. Recently, Vostokov and Orlova \cite{Vostokov-Orlova} gave a generalization of the reciprocity law to all odd $g$. In our situation, their result implies that
\[
\Bigl( \frac{Q}{p} \Bigr)_g = \Bigl( \frac{p}{Q} \Bigr)_g
\]
for all $p \dnd g$.

By the Prime Number Theorem and our restriction on the size of $Q$, we see that
\[
\log Q \asymp y + \log \frac{g}{\text{rad } g}
\]
where rad $g$ denotes the radical of $g$. It follows that there exists a constant $C = C(g)$ such that $y \geq C \, \log Q$. Combining this with (\ref{Eq:CRT}) and the Vostokov-Orlova reciprocity, we deduce that $\displaystyle \Bigl( \frac{p}{Q} \Bigr)_g = \psi(p)$ for all $p \leq C \, \log Q$ relatively prime to $g$. By complete multiplicativity,
\begin{equation}
\label{Eq:FinalEq}
\Bigl( \frac{n}{Q} \Bigr)_g = \psi(n)
\end{equation}
for all $n \leq C \, \log Q$ coprime to $g$. Letting $y$ tend to infinity, we see that $Q$ must also tend to infinity, whence we find infinitely many $Q$ satisfying (\ref{Eq:FinalEq}). This concludes the proof.
\end{proof}

\textsc{Department of Mathematics, University of Toronto,\newline
Toronto, Ontario, Canada}

\textit{Email:} \texttt{leo.goldmakher@utoronto.ca}

\end{document}